\documentclass[reqno]{amsart}
\usepackage{amssymb}
\usepackage[dvips]{epsfig}
\usepackage{graphicx}
\usepackage{color}
\usepackage{amsmath}
\usepackage{amssymb}
\usepackage{amsfonts}
\usepackage{amsthm}
\usepackage{mathrsfs}

\newtheorem{thm}{Theorem}[section]
\newtheorem{lem}[thm]{Lemma}

\theoremstyle{remark}
\newtheorem{rem}{\bf Remark}[section]
\theoremstyle{definition}
\newtheorem{defn}[thm]{Definition}

\numberwithin{equation}{section}
\usepackage[colorlinks=true,pdfstartview=FitV,linkcolor=magenta,citecolor=cyan]{hyperref}

\usepackage{bm}
\begin{document}

\title[ ]{Localized state for  nonlinear  disordered stark model}

\author[Shengqing Hu]{Shengqing Hu}

\address{Faculty of Computational Mathematics and Cybernetics, Shenzhen MSU-BIT University, 518172, Shenzhen, China}  \email{hushengqing@smbu.edu.cn}

\author[Yingte Sun]{Yingte Sun*}
  \thanks{*Corrseponding author}

\address{School of Mathematical Sciences, Yangzhou University, yangzhou, P.R.China} \email{sunyt@yzu.edu.cn}

\keywords{Nonlinear disordered stark model, Localization, KAM theory}
\thanks{{\em 2010 Mathematics Subject Classification.} Primary: 35Q40. Secondary: 37K55.}

\begin{abstract}
In this paper, we consider the following nonlinear disordered Stark model:
$${\bf i}\partial_tu_n+\delta(u_{n+1}+u_{n-1})+nu_n+v_nu_n+\epsilon |u_n|^{2}u_n=0,\quad n\in\mathbb{Z}.$$
By employing the diagonalization of the associated linear operators and the KAM theory for nonlinear Hamiltonian systems, we establish that for parameters $\delta$ and $\varepsilon$ in a reasonable range, and for most realization of random variables
$v=\{v_n\}_{n \in \mathbb{Z}}$, there exist time quasi-periodic and spatially localized states that exhibit arbitrary power-law spatial decay.
 \end{abstract}

\maketitle

\section{Introduction and main results}

Anderson localization is a fundamental concept in condensed matter physics, explaining how disorder can suppress  the diffusion of quantum particles living on lattice. The phenomenon is modelled by a random Schr\"odinger operator $\mathrm{H}$ defined on the lattice $\mathbb{Z}^d$. This operator acts  on a wave function $q$ as follows:
\begin{equation}\label{linear}
(\mathrm{H} q)(n) = -\sum_{|m-n|_1=1} q(m) + V_{\omega}(n) q(n),
\end{equation}
where the on-site potentials ${V_{\omega}(n)}$ are independent and identically distributed random variables. Localization manifests in two complementary ways: spectrally, $\mathrm{H}$ exhibits only pure point spectrum with exponentially decaying eigenfunctions; and dynamically, through the preservation of spatial localization over time. The latter is quantitatively captured by the condition that for an well-localized initial state $q(0)$, the solution to the linear equation \begin{equation}\mathrm{i}\partial_t q = \mathrm{H} q \end{equation} obeys
\begin{equation}
\sup_{t \in \mathbb{R}} \sum_{n \in \mathbb{Z}^d} (1 + |n|_2)^{2p} |q(n)(t)|^2 < \infty,
\end{equation}
demonstrating the  absence of diffusion. Localization can be induced by mechanisms beyond disorder, most notably quasi-periodic potentials. The Mathieu operator \cite{FSW1990} and the Maryland model \cite{BLS1983} are two paradigmatic examples in this class. Their localization properties are intimately tied to-and indeed proven from-the defining geometric features of the potential, such as non-degeneracy or monotonicity.

Another distinct mechanism is Wannier-Stark localization \cite{W1960}, a phenomenon in condensed matter physics where a single particle in a lattice, under a uniform electric field, exhibits super-exponentially localized eigenstates. Physically, this occurs because the field-induced energy gap suppresses the particle's diffusion, which is otherwise facilitated by the lattice hopping terms. Mathematically, it has been established that for such models, the localization properties remain robust against different hopping terms \cite{dp2019,SW2024} or perturbations of arbitrary strength \cite{SW2025,M2025}.

Over the past two decades, there has been growing interest in whether solutions of  linear systems  as \eqref{linear} exhibit richer  behavior under   nonlinear effects. A representative model  is the discrete nonlinear disordered Schr\"odinger equation (DNDSE):
\begin{equation}
\mathrm{i}\partial_t q(n) = -\delta\sum_{|m-n|_1=1} \big(q(m) - q(n)\big) + V_{\omega}(n) q(n) + \epsilon |q(n)|^2 q(n),
\end{equation}
with its quasi-periodic counterpart (DNQSE)\footnote{discrete nonlinear quasi-periodic Schr\"odinger equation} by  a potential such as $V(n)=\cos(\theta+n\cdot \omega)$.
From a physical perspective, these nonlinear models serve as effective approximate models for certain quantum many-body systems and provide crucial insights into Bose-Einstein condensates. It is widely accepted in the physics community  that even weak nonlinearities can cause sub-diffusive transport of wave packet at high energies. Mathematically, investigating the dynamic behavior of localized state under the nonlinear effects is a intriguing problem of nonlinear dynamics.  However, rigorous mathematical results on  sub-diffusive phenomena remain scarce. Most existing mathematical studies are confined to following relatively idealized settings:

 $\bullet$ Investigating whether localization over finite time, albeit long, time scales.  In various parameter regimes and phase space, power-law or sub-exponential long-time localization have been rigorously established. We refer readers to \cite{WZ09,CSZ21,CSW24,FKS08,RHP25}.

 $\bullet$ Constructing special localized states, such as time quasi-periodic (or periodic) localized states. These states correspond to low-energy solutions and are consistent with  the physical picture of inhibited transport.

This work focuses on  the second category of problem: the construction of time quasi-periodic localized states.

For the DNDSEs,  Albanese-Fr\"ohlich-Spencer  \cite{AF1988,AFS1988}  firstly employ multi-scale analysis and bifurcation methods to obtain nonlinear eigenstates, thereby establishing a class of time-periodic localized states. Bourgain-Wang \cite{BW2008} used multi-scale analysis and Newton iteration to construct time quasi-periodic localized states for the high-dimensional DNDSEs. However, their approach required both the hopping term and the nonlinear term to be treated as small perturbations simultaneously.  More recently, Liu-Wang \cite{LW2024} revisited the one-dimensional DNQSE and established the existence of such time-quasi-periodic localized states at arbitrary disorder.\footnote{Equivalently, there is no requirement regarding the size of Laplace.}

Significant progress has also been made in the study of the DNQSEs. Geng-Zhao-You \cite{GZ013,GYZ2014} applied classical infinite-dimensional KAM theory to prove the existence of time quasi-periodic localized states for certain  one-dimensional DNQSEs.  Wang-Shi et al. \cite{SW23, SW24, LSZ2025}  extended these results to several high-dimensional DNQSEs by employing multi-scale analysis.  This work marked a significant breakthrough, as extending localization results to higher dimensions presents substantial additional technical challenges not present in the one-dimensional case.

\subsection{Main result}\

This paper aims to investigate how a static force affects the nonlinear localized state for the DNDSE model? This setting corresponding to Bose-Einstein condensates subject to gravitational field, an  magnetic field gradient,  or
an acceleration of the optical lattice as a whole. To address this problem, we analysis  the  nonlinear  disordered stark model governed by the following equation:
 \begin{equation}\label{eq1.1}
{\bf i}\partial_tu_n+\delta(u_{n+1}+u_{n-1})+nu_n+v_nu_n+\epsilon |u_n|^{2}u_n=0,\quad n\in\mathbb{Z}.
\end{equation}
where $v_n (n\in \mathbb{Z})$ are i.i.d random variables in the interval $[-\frac{1}{10},\frac{1}{10}]$.

\begin{thm}\label{thm1.1}
Consider one-dimensional nonlinear disordered stark model \eqref{eq1.1}. Fix $n_k \in \mathbb{Z}, k=1,\cdots, b$ and $d>0$. Let $\mathcal{J}=\{n_k\}^{b}_{k=1} \in \mathbb{Z}$ and $\mathcal{V}=\{v_{\alpha}\}_{\alpha \in \mathcal{J}} \in \mathbb{R}^b$.
Then for  $0<\delta< \delta_0=O(\frac{1}{b})$, there exits a  positive constant $\epsilon_0(d,b,\mathcal{J})\ll 1$ such that, for any $0<\epsilon<\epsilon_0$, there exists a cantor like set $\mathcal{O}_{\epsilon}\subset \mathcal{O}:=[-\frac{1}{10},\frac{1}{10}]^b$ with
\begin{equation}
\text{meas}( \mathcal{O} \backslash \mathcal{O}_{\epsilon}) \leq C\epsilon^{\frac{1}{16}},
\end{equation}
such that if $\mathcal{V} \in \mathcal{O}_\epsilon $, then equation \eqref{eq1.1} has solution $u(x,t)$ of the form
\begin{equation}\label{solution}
 u(x,t)=\sum_{n\in \mathbb{Z}}u_n(t)\delta_n(x)
\end{equation}
where $\delta_n$ are standard basis for $\mathbb{Z}$ and  $q_n(t)$ are quasi-periodic function with frequencies $\omega=(\omega_1.\cdots,\omega_b)$. Moreover,
\begin{equation}
|\omega-\mathcal{V}|\leq C \epsilon,
\end{equation}
and
\begin{equation}
\sup_{t \in \mathbb{R}}\sum_{n \in \mathbb{Z}}|u_n(t)|^2\langle n \rangle^{2d} <+\infty.
\end{equation}
\end{thm}

\begin{rem}\

$\bullet$ The random variables are assumed in this paper to be uniformly distributed on $[-\frac{1}{10}, \frac{1}{10}]$. Both the results and the proofs, however, remain valid for any absolutely continuous distribution.

$\bullet$ Indeed, by applying the nonlinear KAM theorem (Theorem \ref{main theorem}), the solution to the equation can be expressed in a complete orthogonal basis of the linear operator. Then, using the unitary transformation derived from the KAM iteration of the linear part, the solution can be reformulated in the standard basis ${\delta_n}$.

$\bullet$ We only use finitely many random variables to control the resonance phenomena, while no conditions are imposed on the remaining ones. Consequently, for a fixed realization of the random potential, we cannot obtain a family of time-quasi periodic localized states parameterized by the amplitude of the initial data. Addressing this limitation is a key objective of our subsequent work. Moreover, if this issue is resolved, the constraint on the magnitude of the hopping term can be further relaxed, it would only need to be smaller than a fixed absolute constant. Furthermore, the existence of time almost-periodic localized states could then also be investigated.

$\bullet$ Repeating all the procedures in this paper, it can be seen that the conclusions will not change significantly when considering an exponentially decaying hooping term.

$\bullet$ In fact, the paper ultimately obtains time quasi-periodic localized states with arbitrary power-law decay in space, rather than exponential decay. This is mainly due to technical reasons arising from the extensive use of truncation techniques for measure estimates within the framework of the nonlinear KAM theorem. If a random potential with power-law decay is considered, the truncation could be replaced by alternative techniques, thereby leading to exponentially decaying localized states.

$\bullet$
Physically, the localization and delocalization for the nonlinear Stark model have also motivated many interesting studies; for relevant works, see \cite{GS2009,KKF2009}.

\end{rem}
\subsection{New ingredient of proof}\

$\bullet$
The investigation of nonlinear localized states remains largely confined to the so-called atomic limit (see \cite{RHP25}), where the combined effect of the hopping and nonlinear terms is treated  perturbatively. Therefore, advancing beyond the atomic limit to construct time quasi-periodic localized states presents a problem of significant interest. In this work, for a fixed time-frequency dimension, we fix the strength of the hopping term and treat only the nonlinearity as a perturbation. We rigorously establish the existence of such time-quasi-periodic localized states  when the nonlinear term is sufficiently weak. To the best of our knowledge, only the very recent work  \cite{LW2024} by Liu-Wang has succeeded in transcending the atomic limit. Our analysis of the nonlinear disordered Stark model reaches a comparable conclusion.

Most research efforts remain confined to the so-called atomic limit scenario, largely due to the difficulty in obtaining detailed information about the eigenvalues of the linear part, particularly their parametric dependence, when treating the linear component as a whole. This explains why the presence of the hopping term was not considered in certain early studies \cite{Y2002,FSW1986} of nonlinear localization in lattice systems. Even for models like the well-known Maryland model, which similarly allows uniform diagonalization, this limitation forces the hopping term's magnitude to be strictly dependent on the strength of the nonlinearity.
In the present work, we introduce finite random variables as parameters and rigorously analyze the derivatives of eigenvalues with respect to these random parameters during the KAM diagonalization process. This approach enables our results to transcend the atomic limit.
However, because the magnitude of the hopping term in our eigenvalue differentiation scheme is governed by the dimension of the selected random parameters, we cannot establish the non-degeneracy of the Jacobian matrix for all random potentials. Demonstrating such non-degeneracy would require the Jacobian matrix to exhibit off-diagonal decay, a property that cannot be rigorously verified within our current linear KAM iteration framework.
Consequently, unlike the approach of Liu-Wang \cite{LW2024}, which constructs families of time quasi-periodic localized states for  \textquotedblleft good" random potentials by adjusting internal parameters, our method currently only ensures the existence of individual localized states. This outcome is analogous to the earlier result by Bourgain-Wang \cite{BW2008}. We intend to address this limitation in future research.

$\bullet$
Indeed, constructing time quasi-periodic localized states through multi-scale analysis and Newton iteration offers distinct advantages, most notably a greater tolerance for resonances. While classical KAM theory, particularly approaches based on Hamiltonian normal forms, possesses its own merits(such as enabling the analysis of linear stability of such states, albeit a secondary consideration),recent studies \cite{HSSY2024} have also shown that linear stability can be obtained through multi-scale analysis. Its principal strength, however, lies in furnishing a framework to examine the localization properties of solutions near quasi-periodic localized states. It must be acknowledged that some KAM theorems overlook this aspect, as their normal forms ultimately converge to a single solution during the iteration process.

Returning to Kolmogorov's original idea, we have refined the classical KAM iteration scheme within our nonlinear KAM theorem. This allows us not only to construct time quasi-periodic localized states, but also to derive the corresponding normal form in a neighborhood of the state, thereby establishing a systematic framework for analyzing localization properties in the vicinity of  such states. This idea has in fact already found applications in the study of partial differential equations \cite{CML2015,CLY2016}.

Although the present study does not pursue this direction further, we plan to conduct a more in-depth investigation into the localization properties near these time quasi-periodic localized states in future work, once the aforementioned parameter dependence issues have been resolved.

\subsection{Organization of the paper}\

In Section 2, we perform a KAM diagonalization for the linear part of the equation \eqref{eq1.1}, select finite-dimensional random variables as parameters, and obtain estimates for the derivatives of the eigenvalues with respect to these parameters.
In Section 3, we verify the Hamiltonian structure of the model after the linear part has been diagonalized.
In Section 4, we establish a KAM theorem for the diagonalized nonlinear stark model and show the existence of time quasi-periodic localized states exhibiting power-law decay. In Section 5. we give the detailed proof of the nonlinear KAM theorem.

\section{The eigenvalue variation}
In this section, we study how the eigenvalues of the linear part of Schr\"odinger equation \eqref{eq1.1} depend on the random variable $v_n$,  in order to derive estimates for the derivatives of these eigenvalues with respect to the random variables, as well as the decay estimates of the unitary transformation required for diagonalizing the linear part. This lays the groundwork for establishing the KAM theorem in Section 4.

For convenience, we may regard the linear part of the nonlinear equation \eqref{eq1.1} as a liner Schr\"odinger operator $\mathcal{L}$, namely
\begin{equation}\label{s}
\mathcal{L}=\mathcal{L}_0+V,
\end{equation}
with
$$(\mathcal{L}_0)u_n=\delta(u_{n+1}+u_{n-1})+nu_n,\quad (Vu)_n=v_nu_n.$$

Then, we have the following theorem.
\begin{thm}\label{yubei}
For Schr\"odinger operator \eqref{s}, let $\delta \leq \frac{1}{55}$ and $v=\{v_n\}_{n\in \mathbb{Z}}$ be a family of independent identically distributed (i.i.d) random variables in $\left[-\frac{1}{10},\frac{1}{10}\right]$. Then, the linear operator $\mathcal{L}$ has discrete pure point spectrum. Moreover, there is a unitary transformation $\mathrm{G}:\ell^2(\mathbb{Z})\rightarrow \ell^2(\mathbb{Z})$ such that
\begin{equation}\mathrm{G}^* \mathcal{L} \mathrm{G}={\rm diag}\big\{ d_n: n\in\mathbb{Z}\big\}={\rm diag}\big\{ n+f_n(v): n\in\mathbb{Z}\big\},\end{equation}
\begin{equation}
|d_n-d_m|\geq \frac{2}{3}(m\neq n),\quad \left|\frac{\partial f_n}{\partial v_m}-\delta_{mn}\right|\le \frac{26}{15}C(\delta),\quad \forall m,n\in\mathbb{Z},
 \end{equation}
 \begin{equation}
|(\mathrm{G}-\mathbb{I})_{mn}|+\frac{1}{10}\Big| \frac{\partial(\mathrm{G}-\mathbb{I})_{mn}}{\partial v_j}\Big|\le \frac{19}{9} \mathrm{C}(\delta)e^{-\frac{1}{8}|m-n|},
\end{equation}
where  $\mathrm{C}(\delta)=4\delta e^{\frac{1}{8}}e^{4\delta e^{\frac{1}{8}}}$.
\end{thm}

The diagonalization of such Schr\"odinger operators has been thoroughly investigated in \cite{SW2024}. In this section, we revisit this diagonalization procedure to derive estimates for the derivatives of eigenvalues with respect to the random variable $v_n$, thereby laying the groundwork for the non-degeneracy condition on frequencies in our application of KAM theorem in Section 4.

\begin{defn}\label{defn1}
For a matrix $\mathrm{A}=(\mathrm{A}_{mn})_{m,n\in\mathbb{Z}}$, given $r>0$, we say $\mathrm{A}\in \mathcal{M}_r$ if and only if
$$\Vert \mathrm{A}\Vert_{r}:=\sum_{l\in \mathbb{Z}}e^{r|l|}\sup_{m-n=l}|\mathrm{A}_{mn}|<\infty.$$
\end{defn}

\begin{rem}It is readily verified that if an operator
$\mathrm{A}$ belongs to
$\mathcal{M}_r$, then it is also a bounded operator on the standard
$\ell^2$ space. Consequently, self-adjoint and unitary operators can be defined with respect to the canonical inner product on $\ell^2$.
\end{rem}
\begin{defn}\label{defn22}
For $0<\alpha<1$, If an operator is differentiable about the parameter $v\in \Pi:=\left[-\frac{1}{10},\frac{1}{10}\right]^{\mathbb{Z}}$, we say $\mathrm{A}\in \mathcal{M}^\alpha_r$ if and only if
$$\Vert \mathrm{A}\Vert^\alpha_{r}=\sup_{v\in \Pi}\Vert \mathrm{A}\Vert_{r}+\alpha\sup_{v\in\Pi}\left\Vert \frac{\partial \mathrm{A}}{\partial v}\right\Vert_{r}<\infty,$$
where
$$\left\Vert   \frac{\partial \mathrm{A}}{\partial v}\right\Vert_r:=\sup_{j\in\mathbb{Z}}\left\Vert   \frac{\partial \mathrm{A}}{\partial v_j}\right\Vert_r.$$
\end{defn}

\begin{lem}\label{chen}
Let $\mathrm{A},\mathrm{B}\in\mathcal{M}_r$. Then $\mathrm{AB}\in \mathcal{M}_r$ and
$$\Vert \mathrm{AB}\Vert_r\le \Vert \mathrm{A}\Vert_r\Vert \mathrm{B}\Vert_r.$$
Let $\mathrm{A}, \mathrm{B}\in\mathcal{M}^\alpha_r$. Then $\mathrm{A}\mathrm{B}\in \mathcal{M}^\alpha_r$ and
\begin{align}\label{ral}
\Vert \mathrm{AB}\Vert^\alpha_r\le \Vert \mathrm{A}\Vert^\alpha_r\Vert \mathrm{B}\Vert^\alpha_r.
\end{align}
\end{lem}
\begin{proof}
By definition \ref{defn1}, one has
\begin{align*}
\Vert \mathrm{AB}\Vert_{r}=&\sum_{l\in \mathbb{Z}}e^{r|l|}\sup_{m-n=l}|(\mathrm{AB})_{mn}|\\
=&\sum_{l\in \mathbb{Z}}e^{r|l|}\sup_{m-n=l}|\sum_{k\in\mathbb{Z}}\mathrm{A}_{mk}\mathrm{B}_{kn}|\\
\le &\sum_{l,k\in \mathbb{Z}} e^{r|l|} \sup_{m-n=l}|\mathrm{A}_{mk}\mathrm{B}_{kn}|\\
\le &\sum_{l,j\in \mathbb{Z}} e^{r(|j|+|l-j|})\sup_{m-k=j}|\mathrm{A}_{mk}|\sup_{k-n=l-j}|\mathrm{B}_{kn}|\\
\le &\Vert \mathrm{A}\Vert_{r} \Vert \mathrm{B}\Vert_{r}.
\end{align*}
Similarly, one can show \eqref{ral}.
\end{proof}

 Let
 $$\mathrm{D}={\rm diag}\big\{n:n\in \mathbb{Z}\big\},\quad \Delta=\left\{\begin{array}{ll}
1,\quad &n-m=\pm 1,\\
 0,\quad &{\rm otherwise}.
 \end{array}
 \right.$$
 Then, one has
 $$\mathcal{L}_0=\mathrm{D}+\delta \Delta.$$
 In the following, let $r=\frac{1}{8}$, then $\Vert \Delta\Vert_{\frac{1}{8}}=2e^{\frac{1}{8}}$.

\begin{lem}\label{Uest}
There is a unitary transformation $\mathrm{U}\in \mathcal{M}_{\frac{1}{8}}$ such that
$$\mathrm{U}^*\mathcal{ L}_0\mathrm{U}=\mathrm{D},$$
$$\quad \Vert \mathrm{U}\Vert_{\frac{1}{8}}\le e^{\delta\Vert \Delta\Vert_{\frac{1}{8}}}=e^{2 \delta e^{\frac{1}{8}}},$$
\begin{equation*}
\Vert \mathrm{U}-\mathbb{I}\Vert_{\frac{1}{8}} \leq  2\delta e^{\frac{1}{8}}e^{2\delta e^{\frac{1}{8}}}.
\end{equation*}
\end{lem}
 \begin{proof}
 Set $\mathrm{U}=e^{-\delta \mathrm{W}}$, then one has
 \begin{align}\label{1}
 e^{\delta \mathrm{W}}\mathcal{L}_0e^{-\delta \mathrm{W}}=&\mathrm{D}+[\delta \mathrm{W},\mathrm{D}]+\delta \Delta+\sum_{n=2}^\infty \frac{\mathbf{ad}^n_{\delta \mathrm{W}}(\mathrm{D})}{n!}+\sum_{n=1}^\infty \frac{\mathbf{ad}^n_{\delta \mathrm{W}}(\delta\Delta)}{n!}.
 \end{align}
 In the following, we will solve the equation
\begin{equation} \label{2}
 [\delta \mathrm{W},\mathrm{D}]+\delta \Delta=0.
 \end{equation}
By computing the matrix element of the above equation, one can  get
 \begin{align}\label{to}
 \mathrm{W}_{mn}=\frac{\Delta_{mn}}{m-n}, \quad \forall m,n \in \mathbb{Z},
 \end{align}
 which implies that
$$ \mathrm{W}_{mn}=\left\{\begin{array}{ll}
 \frac{1}{m-n},\quad &n-m=\pm 1,\\
\ \ 0, &{\rm otherwise}.
 \end{array}
 \right.$$
Finally, one has
 $$\Vert \mathrm{W}\Vert_{\frac{1}{8}}=\Vert \Delta\Vert_{\frac{1}{8}}=2e^{\frac{1}{8}}.$$
Obviously, $\Delta$ is a T\"oplitz matrix, by \eqref{to}, one sees that $\mathrm{W}$ is also a T\"oplitz matrix. Since the commutator of two T\"oplitz matrices vanishes( see the Lemma 2.6 in \cite{HS2025}), one has
  \begin{align}\label{yao}
\mathbf{ad}^{n}_{ \delta \mathrm{W}}(\delta\Delta)=0,\quad n\geq 1.
 \end{align}
 and
 \begin{equation}\label{3}
 \mathbf{ad}^n_{ \delta \mathrm{W}}(\mathrm{D})=\mathbf{ad}^{n-1}_{\delta \mathrm{W}}(-\delta\Delta)=0,\quad n\geq 2.
 \end{equation}
 Substituting \eqref{2}, \eqref{yao} and \eqref{3} into \eqref{1}, one  gets
 $$e^{\delta \mathrm{W}}\mathcal{L}_0e^{- \delta \mathrm{W}}=\mathrm{D},$$
where
\begin{equation}\Vert \mathrm{U}\Vert_{\frac{1}{8}}=\Vert e^{-\delta \mathrm{W}}\Vert_{\frac{1}{8}}\le e^{\delta \Vert \mathrm{W}\Vert_{\frac{1}{8}}}\le e^{2\delta e^{\frac{1}{8}}},\footnote{It is easy to show that $e^{2\delta e^{\frac{1}{8}}} \leq \frac{21}{20}$, when $\delta \leq \frac{1}{55}$. }\end{equation}
\begin{equation}
\begin{split}
\Vert \mathrm{U}-\mathbb{I}\Vert_{\frac{1}{8}} &\leq \sum_{n=1}^\infty  \frac{(\delta\Vert  \mathrm{W}\Vert_{\frac{1}{8}})^n}{n!}\leq \delta \|\mathrm{W}\|_{\frac{1}{8}} \sum_{n=1}^\infty \frac{(\delta\Vert  \mathrm{W}\Vert_{\frac{1}{8}})^{n-1}}{n!} \\
&\leq \delta \|\mathrm{W}\|_{\frac{1}{8}} e^{\delta\|\mathrm{W}\|_{\frac{1}{8}}}\\
&\leq 2\delta e^{\frac{1}{8}}e^{2\delta e^{\frac{1}{8}}}.
\end{split}
\end{equation}
 \end{proof}

It is straightforward to verify that the operator $\mathrm{W}$ is skew-adjoint, while $\mathrm{U}=e^{-\delta\mathrm{W}}$ is unitary. Consequently, the operator
$\mathcal{L}$ is conjugate to

\begin{align}
 \mathcal{L}_{new}=&\mathrm{D}+e^{\delta \mathrm{W}}\mathcal{V}e^{-\delta \mathrm{W}}\nonumber\\
 =&\mathrm{D}+\mathcal{V}+\sum_{n=1}^\infty  \frac{\mathbf{ad}^n_{\delta \mathrm{W}}(\mathcal{V})}{n!}\nonumber\\
 =&\mathrm{D}_{new}+\mathrm{P}_{new},\label{zhuy}
  \end{align}
 where
 $$\mathrm{D}_{new}=\mathrm{D}+\mathcal{V},\quad \mathrm{P}_{new}=\sum_{n=1}^\infty  \frac{\mathbf{ad}^n_{\delta \mathrm{W}}(\mathcal{V})}{n!}.$$

 It follows that
 \begin{align}
 \Vert  \mathrm{P}_{new}\Vert_{\frac{1}{8}}\le &\sum_{n=1}^\infty  \frac{\Vert \mathbf{ad}^n_{\delta \mathrm{W}}(\mathcal{V})\Vert_{\frac{1}{8}}}{n!}
 \le \sum_{n=1}^\infty  \frac{(2\delta\Vert  \mathrm{W}\Vert_{\frac{1}{8}})^n\Vert \mathcal{V} \Vert_{\frac{1}{8}}}{n!}\nonumber\\
 \le & 2\delta \Vert  \mathrm{W}\Vert_{\frac{1}{8}}\Vert \mathcal{V} \Vert_{\frac{1}{8}} \sum_{n=1}^\infty  \frac{(2\delta\Vert  \mathrm{W}\Vert_{\frac{1}{8}})^{n-1}}{n!}\nonumber\\
 \le &  2\delta \Vert  \mathrm{W}\Vert_{\frac{1}{8}}\Vert \mathcal{V} \Vert_{\frac{1}{8}} e^{2\delta \Vert  \mathrm{W}\Vert_{\frac{1}{8}}}\\
 \le & 4\delta e^{\frac{1}{8}}e^{4\delta e^{\frac{1}{8}}}\Vert \mathcal{V} \Vert_{\frac{1}{8}} .\label{epnew}
 \end{align}
 Let $\mathrm{C}(\delta)=4\delta e^{\frac{1}{8}}e^{4\delta e^{\frac{1}{8}}}$ \footnote{A simple calculation shows that $\mathrm{C}(\delta)<\frac{1}{11}$ holds when $\delta < \frac{1}{55}$.},  since $\mathrm{W}$ is independent of $v_n(n\in\mathbb{Z})$, one has
 \begin{align}
\left\Vert   \frac{\partial \mathrm{P}_{new}}{\partial v}\right\Vert_{\frac{1}{8}}&=\sup_{m\in\mathbb{Z}}\left\Vert   \frac{\partial \mathrm{P}_{new}}{\partial v_m}\right\Vert_{\frac{1}{8}}\nonumber\\
\le &\sup_{m\in\mathbb{Z}}\sum_{n=1}^\infty  \frac{\left\Vert \mathbf{ad}^n_{\delta \mathrm{W}}(\frac{\partial \mathcal{V}}{\partial v_m})\right\Vert_{\frac{1}{8}}}{n!}
 \le \sum_{n=1}^\infty  \frac{(2\delta\Vert  \mathrm{W}\Vert_{\frac{1}{8}})^n\sup_{m\in\mathbb{Z}}\left\Vert \frac{\partial \mathcal{V}}{\partial v_m} \right\Vert_{\frac{1}{8}}}{n!}\nonumber\\
 \le & 2\delta \Vert  \mathrm{W}\Vert_{\frac{1}{8}}\sup_{m\in\mathbb{Z}}\left\Vert \frac{\partial \mathcal{ V}}{\partial v_m} \right\Vert_{\frac{1}{8}} \sum_{n=1}^\infty  \frac{(2\delta\Vert  \mathrm{W}\Vert_{\frac{1}{8}})^{n-1}}{n!}\nonumber\\
 \le &  2\delta \Vert  \mathrm{W}\Vert_{\frac{1}{8}}\sup_{m\in\mathbb{Z}}\left\Vert \frac{\partial \mathcal{V}}{\partial v_m} \right\Vert_{\frac{1}{8}} e^{2\delta \Vert  \mathrm{W}\Vert_{\frac{1}{8}}}\\
 \le & \mathrm{C}(\delta).\label{epnew1}
 \end{align}



Next, we will diagonalize the operator $\mathcal{L}_{new}$ via a KAM-type iteration. Since the process does not involve small divisors, the iteration becomes straightforward. Instead of merely assuming the perturbation is sufficiently small, we explicitly quantify its size. This clarifies the dependence of perturbation magnitudes in constructing quasi-periodic solutions-an essential feature of our analysis.

At first, we state two key quantitative lemmas.
 \begin{lem}\label{jie}
 Let $\mathrm{D}$ be a diagonal operator with $\mathrm{D}_{nn}=\mathrm{d}_n\in\mathbb{R}$ and
 \begin{equation}\label{cond}
 |\mathrm{d}_n-\mathrm{d}_m|\geq \frac{2}{3},\quad  |\partial_{v_j}(\mathrm{d}_n-\mathrm{d}_m)|\le \frac{3}{2},\quad n\neq m.
 \end{equation}
 Given $\mathrm{P}\in\mathcal{M}^\alpha_r$, there exists $\mathrm{W}\in\mathcal{M}^\alpha_r$ solving
 \begin{equation}\label{hoo}
 [\mathrm{W},\mathrm{D}]+\mathrm{P}-{\rm diag} \mathrm{P}=0
 \end{equation}
 and satisfying
 $$\Vert \mathrm{W}\Vert^\alpha_r\le \frac{39}{8} \Vert \mathrm{P}\Vert^\alpha_r.$$
Moreover, if $\mathrm{P}$ is a self-adjoint operator, then $\mathrm{W}$ is an skew-adjoint operator.
 \end{lem}

 \begin{proof}
 For the linear equation \eqref{hoo}, one gets
 $$\mathrm{W}_{mn}\mathrm{d}_n-\mathrm{d}_m\mathrm{W}_{mn}+\mathrm{P}_{mn}=0,$$
and
 $$ \mathrm{W}_{mn}=\left\{\begin{array}{ll}
 \frac{\mathrm{P}_{mn}}{\mathrm{d}_m-\mathrm{d}_n},\quad &n\neq m,\\
\ \ 0, &n=m.
 \end{array}
 \right.$$
 From the bounds \eqref{cond}, one has
 $$|\mathrm{W}_{mn}|=\left| \frac{\mathrm{P}_{mn}}{\mathrm{d}_m-\mathrm{d}_n}\right|\le \frac{3}{2}|P_{mn}|,$$
 and
 \begin{equation}
 \begin{split}
 \left| \frac{\partial \mathrm{W}_{mn}}{\partial_{v_j}}\right|&=\left| \frac{1}{d_m-d_n}\frac{\partial \mathrm{P}_{mn}}{\partial v_j}+\frac{\partial (\mathrm{d}_{m}-\mathrm{d}_n)}{\partial v_j}\frac{\mathrm{P}_{mn}}{(\mathrm{d}_m-\mathrm{d}_n)^2}\right|\\
 &\le \frac{3}{2}\left| \frac{\partial \mathrm{P}_{mn}}{\partial v_j}\right|+\frac{27}{8}|  \mathrm{P}_{mn}|.
 \end{split}
 \end{equation}
 Since $0<\alpha<1$, one has
 \begin{equation}
 \begin{split}
 \Vert \mathrm{W}\Vert^\alpha_r\le& \left(\frac{3}{2}+\frac{27}{8} \right)\Vert \mathrm{P}\Vert_r+ \frac{3}{2} \alpha \left\Vert \frac{\partial \mathrm{P}}{\partial v}\right\Vert_r\\
   \le& \frac{39}{8}\Vert \mathrm{P}\Vert^\alpha_r.
  \end{split}
 \end{equation}
 Since the operator $\mathrm{P}$ is self adjoint, one has $\mathrm{P}_{nm}=\overline{\mathrm{P}_{mn}}.$ Then, we obtain
 $$\mathrm{W}_{mn}= \frac{\mathrm{P}_{mn}}{\mathrm{d}_m-\mathrm{d}_n}=- \frac{\overline{\mathrm{P}_{nm}}}{\mathrm{d}_n-\mathrm{d}_m}=-\overline{\mathrm{W}_{nm}},$$
 which indicates $\mathrm{W}$ is skew-adjoint.
 \end{proof}

\begin{lem}\label{+}
For the linear operator
$$\mathcal{L}=\mathrm{D}+\mathrm{P},$$
where $\mathrm{D}$ satisfies the assumptions in Lemma \ref{jie} and $\mathrm{P}\in \mathcal{M}^\alpha_r$. 
 Then there is a unitary transformation $e^{\mathrm{W}}\in \mathcal{M}^\alpha_r$ such that
$$e^{\mathrm{W}} (\mathrm{D}+\mathrm{P})e^{-\mathrm{W}}=\mathrm{D}_++\mathrm{P}_+,$$
where $\mathrm{D}_+=\mathrm{D}+{\rm diag} \mathrm{P}$ and $\mathrm{P}_+\in \mathcal{M}^\alpha_r$ with the following estimate
$$\Vert \mathrm{P}_+\Vert^\alpha_r\le  \frac{117}{8}(\Vert \mathrm{P}\Vert^\alpha_r)^2 e^{\frac{78}{8}\Vert \mathrm{P}\Vert^\alpha_r}.$$
\end{lem}

\begin{proof}
From Lemma \ref{jie}, we have $ [\mathrm{W},\mathrm{D}]+\mathrm{P}={\rm diag} \mathrm{P}$. Therefore, we have
\begin{align*}
e^{\mathrm{W}}\mathcal{L}e^{-\mathrm{W}}=&\mathrm{D}+[\mathrm{W},\mathrm{D}]+\mathrm{P}+\sum_{n=2}^\infty \frac{\mathbf{ad}^n_{\mathrm{W}}(\mathrm{D})}{n!}+\sum_{n=1}^\infty \frac{\mathbf{ad}^n_{\mathrm{W}}\mathrm{P})}{n!}\\
=&\mathrm{D}+{\rm diag} \mathrm{P}+\sum_{n=2}^\infty \frac{\mathbf{ad}^{n-1}_\mathrm{W}({\rm diag} \mathrm{P}-\mathrm{P})}{n!}+\sum_{n=1}^\infty \frac{\mathbf{ad}^n_{\mathrm{W}}(\mathrm{P})}{n!}.
\end{align*}
Let $\mathrm{D}_+=\mathrm{D}+{\rm diag} \mathrm{P}$ and
$$\mathrm{P}_+=\sum_{n=2}^\infty \frac{\mathbf{ad}^n_W({\rm diag} \mathrm{P}-\mathrm{P})}{n!}+\sum_{n=1}^\infty \frac{\mathbf{ad}^n_\mathrm{W}(\mathrm{P})}{n!}.$$
From definition \ref{defn22}, one has
$$\Vert {\rm diag} \mathrm{P}- \mathrm{P}\Vert^\alpha_r\le \Vert \mathrm{P}\Vert^\alpha_r.$$
From Lemmata \ref{chen}, \ref{jie}, one sees
\begin{align*}
\Vert \mathrm{P}_+\Vert^\alpha_r\le &\sum_{n=2}^\infty \frac{\Vert \mathbf{ad}^{n-1}_{\mathrm{W}}({\rm diag} \mathrm{P}-\mathrm{P})\Vert^\alpha_r}{n!}+\sum_{n=1}^\infty \frac{\Vert\mathbf{ad}^n_\mathrm{W}(\mathrm{P})\Vert^\alpha_r}{n!}\\
\le& \sum_{n=2}^\infty \frac{(2\cdot \frac{39}{8}\Vert \mathrm{P}\Vert^\alpha_r)^{n-1}\Vert \mathrm{P}\Vert^\alpha_r}{n!}+\sum_{n=1}^\infty \frac{(2\cdot \frac{39}{8}\Vert \mathrm{P}\Vert^\alpha_r)^{n}\Vert \mathrm{P}\Vert^\alpha_r}{n!}\\
\le &2\cdot \frac{39}{8}(\Vert \mathrm{P}\Vert^\alpha_r)^2 \sum_{n=2}^\infty \frac{(2\cdot \frac{39}{8}\Vert \mathrm{P}\Vert^\alpha_r)^{n-2}}{n!}+ 2\cdot \frac{39}{8}(\Vert \mathrm{P}\Vert^\alpha_r )^2\sum_{n=1}^\infty \frac{(2\cdot \frac{39}{8}\Vert \mathrm{P}\Vert^\alpha_r)^{n-1}}{n!}\\
\le &\frac{39}{8}(\Vert \mathrm{P}\Vert^\alpha_r)^2 e^{2\cdot \frac{39}{8}\Vert \mathrm{P}\Vert^\alpha_r}+ 2\cdot \frac{39}{8}(\Vert \mathrm{P}\Vert^\alpha_r)^2 e^{2\cdot \frac{39}{8}\Vert \mathrm{P}\Vert^\alpha_r}\\
\le &\frac{117}{8}(\Vert \mathrm{P}\Vert^\alpha_r)^2 e^{\frac{78}{8}\Vert \mathrm{P}\Vert^\alpha_r}.
\end{align*}
\end{proof}

In what follows, set
$$\mathcal{L}^0=\mathrm{D}^0+\mathrm{P}^0,\quad \mathrm{D}^0=\mathrm{D}_{new},\quad \mathrm{P}^0=\mathrm{P}_{new}.$$
Let $\alpha:= \frac{1}{10}$ and $\delta \leq \frac{1}{55}$,  then, by \eqref{epnew} and \eqref{epnew1}, we obatin
\begin{equation}
\begin{split}
\Vert\mathrm{P}^0\Vert^\alpha_{\frac{1}{8}} &\leq \mathrm{C}(\delta)\frac{1}{10}+\mathrm{C}(\delta)\alpha\\
&\leq 2 \mathrm{C}(\delta)\alpha.
\end{split}
\end{equation}
For any  $n,m\in\mathbb{Z}$ with $n\neq m$,
\begin{equation}\big|\partial_{v_j}\big((\mathrm{D}^0)_{nn}-(\mathrm{D}^0)_{mm}\big)\big|= |\partial_{v_j}(n+v_n-(m+v_m))|\le 1,\quad j\in\mathbb{Z},
\end{equation}
and
\begin{equation}
\begin{split}
 |(\mathrm{D}^0)_{nn}-(\mathrm{D}^0)_{mm}|&= |n+v_n-(m+v_m)|\\ &\geq |n-m|-2\alpha
                        \geq \frac{4}{5}.
\end{split}
\end{equation}

\begin{thm}\label{hou}Let $$\epsilon_0:=2\mathrm{C}(\delta)\alpha \leq \frac{1}{55}.$$
Suppose that $\Vert \mathrm{P}^0\Vert^\alpha_{\frac{1}{8}}\leq \epsilon_0$.Then, there exist sequence of diagonal operators $\{\mathrm{D}^{k}\}_{k=0}^\infty$ and bounded operators $$\{\mathrm{W}^{k+1}\}_{k=0}^\infty, \quad \{\mathrm{P}^{k}\}_{k=0}^\infty,$$ such that for each $k \geq 0$, the operators $\mathrm{W}^{k+1}$ and $\mathrm{P}^k$ belong to the space $\mathcal{M}^\alpha_{\frac{1}{8}}$, and the following relations hold:
$$e^{\mathrm{W}^{k+1}}(\mathrm{D}^k+\mathrm{P}^k)e^{-\mathrm{W}^{k+1}}=\mathrm{D}^{k+1}+\mathrm{P}^{k+1},$$
$$\mathrm{U}^k:=\mathrm{e}^{\mathrm{W}^k}\circ \cdots \circ\mathrm{e}^{\mathrm{W}^1},$$
$$\mathrm{D}^{k+1}=\mathrm{D}^k+{\rm diag} \mathrm{P}^k,\quad \mathrm{P}^{k+1}=(\mathrm{P}^k)_+.$$
Moreover, the following bounds hold
\begin{equation}\label{pp}
\Vert \mathrm{P}^k\Vert^\alpha_{\frac{1}{8}} \le \epsilon_k:=\epsilon_0^{(\frac{5}{4})^k}
\end{equation}
\begin{equation}\label{pp1}
\Vert\mathrm{W}^{k+1}\Vert^\alpha_{\frac{1}{8}} \le \frac{39}{8}\epsilon_k,
\end{equation}
\begin{equation}\label{dd}
\mathrm{D}^k={\rm diag} \big\{\mathrm{d}_n^k :  n\in\mathbb{Z}\big \}, \quad |\mathrm{d}_n^k-\mathrm{d}_m^k|\geq \frac{2}{3}+\frac{1}{3^{k+2}},
\end{equation}

\begin{equation}\label{par}
 \left|\frac{\partial d_m^k}{\partial v_n}-\delta_{mn}\right|\le \frac{26}{15}\mathrm{C}(\delta), \quad \forall  m,n \in \mathbb{Z}.
\end{equation}
\end{thm}

\begin{proof}
By assumption, the bounds \eqref{pp}-\eqref{par} hold trivially for $k=0$. Proceeding inductively, we assume the KAM iteration holds at the $k-$th step and verify the corresponding results for the $(k+1)$-th step.
 From the bound  \eqref{dd}, one has
 $$|d_{n}^k-d_m^k|\geq \frac{2}{3}.$$
  From Lemma \ref{jie}, we can obtain the transformation $\mathrm{W}^{k+1}\in\mathcal{M}_{\frac{1}{8}}^\alpha$ with
 $$\Vert \mathrm{W}^{k+1}\Vert_{\frac{1}{8}}^\alpha\le \frac{39}{8}\Vert  \mathrm{P}^{k}\Vert^\alpha_{\frac{1}{8}}\le \frac{39}{8}\epsilon_k.$$
 Since $18<55^{\frac{3}{4}}$, from Lemma \ref{+}, one has
 \begin{align*}
 \Vert \mathrm{P}^{k+1}\Vert_{\frac{1}{8}}^\alpha= &\Vert (\mathrm{P}^{k})_+\Vert_{\frac{1}{8}}^\alpha\le  \frac{117}{8}(\Vert \mathrm{P}^k\Vert^\alpha_{\frac{1}{8}})^2 e^{\frac{78}{8}\Vert \mathrm{P}^k\Vert_{\frac{1}{8}}^\alpha}\\
 \le & \frac{117}{8}e^{\frac{39}{220}} \epsilon_k^2
 \le  18 \epsilon_k^2 \\
 \le & \epsilon_k^{\frac{5}{4}} \leq \epsilon_0^{(\frac{5}{4})^{k+1}}.
  \end{align*}

   Since $\epsilon_0< \frac{1}{55}<\frac{1}{3^3}$,  $\epsilon^{\frac{1}{4}}_1 \leq (\frac{1}{55})^{\frac{5}{16}} < \frac{1}{3} $, one can get
   \begin{align*}
   |\mathrm{d}^{k+1}_n-\mathrm{d}^{k+1}_m|= &|\mathrm{d}^{k}_n-\mathrm{d}^{k}_m+(\mathrm{P}^k)_{nn}-(\mathrm{P}^k)_{mm}|\\
   \geq& |\mathrm{d}^{k}_n-\mathrm{d}^{k}_m|-|(\mathrm{P}^k)_{nn}-(\mathrm{P}^k)_{mm}|\\
   \geq&\frac{2}{3}+(\frac{1}{3})^{k+2}-2\epsilon_k\\
   \geq&\frac{2}{3}+(\frac{1}{3})^{k+3}.
   \end{align*}
   Moreover, from \eqref{epnew1} and \eqref{pp}, we have
   for $m,n \in\mathbb{Z}$,
      \begin{align*}
\left|\frac{\partial d_{m}^{k+1}}{\partial v_n}-\delta_{mn}\right|\le & \sum_{i=0}^{k} \left|\frac{\partial (\mathrm{P}^i)_{mm}}{\partial v_n}\right|
\le \mathrm{C}(\delta)+\sum^k_{i=1}\frac{\epsilon^{(\frac{5}{4})^{i}}_0}{\alpha} \\
\le & \mathrm{C}(\delta)+\alpha^{\frac{1}{4}}(\mathrm{C}(\delta))^{\frac{5}{4}}+\sum^{k}_{i=2}\frac{(\alpha \mathrm{C}(\delta))^{(\frac{5}{4})^{i}}}{\alpha}\\
\le &\frac{7}{5}\mathrm{C}(\delta)+\sum^{k}_{i=2}\alpha^{\frac{9}{16}}\mathrm{C}(\delta)^{(\frac{5}{4})^i}\\
\le &\frac{7}{5} \mathrm{C}(\delta)+\frac{1}{3}\sum^{k}_{i=1}(\frac{1}{2})^{i}\mathrm{C}(\delta)\\
\le & \frac{26}{15}\mathrm{C}(\delta),
   \end{align*}
   since $\alpha\leq \frac{1}{10}$ and $\mathrm{C}(\delta)< \frac{1}{11}$.
\end{proof}

{\bf Proof of Theorem \ref{yubei}}:
\begin{proof}
Firstly, we will show the convergence of $\mathrm{U}^k, (\mathrm{U}^{k})^{-1}$. From Theorem \ref{hou}, one has
\begin{align*}
\Vert e^{\mathrm{W}_{k+1}}-\mathbb{I}\Vert_{\frac{1}{8}}^\alpha\le& \sum_{j=1}^\infty\frac{\Vert (\mathrm{W}^{k+1})^j\Vert_{\frac{1}{8}}^\alpha}{j!}\le \sum_{j=1}^\infty\frac{(\Vert \mathrm{W}^{k+1}\Vert^\alpha_1)^j}{j!}\\
\le & \sum_{j=1}^\infty\frac{(\frac{39}{8}\epsilon_k)^j}{j!}\\
\le &\frac{13}{2}\epsilon_k.
\end{align*}
The, for any $k\geq 0$, one gets
$$\mathrm{U}^{k+1}=e^{\mathrm{W}^{k+1}}\mathrm{U}^k=\mathrm{U}^k+(e^{\mathrm{W}_{k+1}}-\mathbb{I})\mathrm{U}^k$$
and
\begin{align}\label{Uinf}
\Vert \mathrm{U}^{k+1}\Vert_{\frac{1}{8}}^\alpha\le (1+\frac{13}{2}\epsilon_k)\Vert\mathrm{U}^{k}\Vert_{\frac{1}{8}}^\alpha\le \prod_{j=0}^k(1+\frac{13}{2}\epsilon_j)\le \frac{4}{3},
\end{align}
Since
\begin{equation}
\begin{split}
\ln{\prod_{j=0}^k(1+\frac{13}{2}\epsilon_j)}&=\sum^{k}_{j=0}\ln{(1+\frac{13}{2}\epsilon_j)} \leq \sum^{k}_{j=0}\frac{13}{2}\epsilon_j\\
&\leq \frac{65}{6}\epsilon_0 \leq \frac{13}{66}.
\end{split}
\end{equation}
Also, one has
\begin{align*}
\Vert \mathrm{U}^{k+j}-\mathrm{U}^{k}\Vert_{\frac{1}{8}}^\alpha\le& \sum_{i=0}^{j-1} \Vert \mathrm{U}^{k+i+1}-\mathrm{U}^{k+i}\Vert_{\frac{1}{8}}^\alpha\le  \sum_{i=0}^{j-1}\Vert e^{\mathrm{W}^{k+i+1}}-\mathbb{I}\Vert_{\frac{1}{8}}^\alpha \Vert \mathrm{U}^{k+i}\Vert_{\frac{1}{8}}^\alpha\\
\le & \frac{4}{3}\sum_{i=0}^{j-1} \frac{13}{2}\epsilon_{k+i}\le \frac{130}{9} \epsilon_{k}\rightarrow 0,
\end{align*}
which implies $\mathrm{U}^k$ converges to an operator $\mathrm{U}^\infty$ with
\begin{equation}\label{Uinf2}\Vert \mathrm{U}^\infty-\mathbb{I}\Vert_{\frac{1}{8}}^\alpha\le \frac{130}{9}\epsilon_0.
\end{equation}
Let $\mathrm{G}=\mathrm{U}\circ \mathrm{U}^\infty$, from Lemma $\ref{Uest}$ and \eqref{Uinf},\eqref{Uinf2},  one has
\begin{equation}
\begin{split}
\|\mathrm{G}-\mathbb{I}\|^{\alpha}_{\frac{1}{8}}&\leq\|\mathrm{U}\circ \mathrm{U}^\infty-\mathbb{I}\|^{\alpha}_{\frac{1}{8}}\\
&\leq  \|\mathrm{U}\circ \mathrm{U}^\infty-\mathrm{U}^{\infty}\|^{\alpha}_{\frac{1}{8}}+\| \mathrm{U}^{\infty}-\mathbb{I}\|^{\alpha}_{\frac{1}{8}}\\
&\leq  \|\mathrm{U}-\mathbb{I}\|^{\alpha}_{\frac{1}{8}}\cdot \|\mathrm{U}^\infty\|^{\alpha}_{\frac{1}{8}}+\| \mathrm{U}^\infty-\mathbb{I}\|^{\alpha}_{\frac{1}{8}}\\
& \leq\frac{\mathrm{C}(\delta)}{2}\cdot \frac{4}{3}+ \frac{130}{9}\epsilon_0\\
&\leq \frac{19}{9}\mathrm{C}(\delta)
\end{split}
\end{equation}
 The same estimation holds for the $(\mathrm{U}^k)^{-1}$, $(\mathrm{U}^\infty)^{-1}$ and $\mathrm{G}^*:=\mathrm{G}^{-1}=\mathrm{U}^{-1}\circ (\mathrm{U}^\infty)^{-1}$. Finally, one has
$$\mathrm{G}^* \mathcal{L} \mathrm{G}=\mathrm{D}^\infty,$$
with $\mathrm{D}^{\infty}={\rm diag} \big\{\mathrm{d}^{\infty}_n:  n\in\mathbb{Z}\big\}$, where
$$\mathrm{d}^{\infty}_n=\lim_{k\rightarrow \infty}\mathrm{d}_n^k.$$
In addition, one gets
\begin{align*}
|d_n^\infty-n|\le \sum_{k=0}^\infty |\mathrm{P}^k_{nn}|\le \sum_{k=0}^\infty \Vert \mathrm{P}^k\Vert_{\frac{1}{8}} \le \sum_{k=0}^\infty \epsilon_k\le \frac{5}{3}\epsilon_0
\end{align*}
and
$$
 \left|\frac{\partial \mathrm{d}_n^
\infty}{\partial v_m}-\delta_{jn}\right|\le \frac{26}{15}\mathrm{C}(\delta),\quad  \forall m, n\in\mathbb{Z}.$$
\end{proof}

\section{The Hamiltonian structure}
The main objective of this section is to prove that after diagonalizing the linear part, the new nonlinear equation remains a Hamiltonian system, and to derive its Hamiltonian formalism. Although this may seem intuitively natural, we provide a rigorous computational procedure in this section as the foundation for the subsequent KAM theorem.

From the proceeding section, one knows that there exists a real and  unitary transformation $\mathrm{G}$, that conjugate $\mathcal{L}$ into
\begin{equation}
\mathrm{G^*\mathcal{L}G}:=\mathrm{D}=\mathrm{diag}\big\{d_n=n+f_n: n\in \mathbb{Z}\big\}.
\end{equation}
Let $u,q\in \ell^2(\mathbb{Z})$, the discrete Schr\"odinger equation \eqref{eq1.1} can be represented as
\begin{equation}\label{ds}
\mathbf{i}\partial_{t}u+\mathcal{L}u+\epsilon\mathcal{N}(u,\bar{u})=0,
\end{equation}
where $\mathcal{N}(u,\bar{u})$ represents the nonlinear term of equation  \eqref{eq1.1}.

By introducing the coordinate transformation $$u=\mathrm{G}q, \quad \bar{u}=\bar{\mathrm{G}}\bar{q},$$
one gets
\begin{equation*}
\mathbf{i}\mathrm{G}\partial_{t}q+\mathcal{L}\mathrm{G}q+\epsilon\mathcal{N}(\mathrm{G}q,\bar{\mathrm{G}}\bar{q})=0
\end{equation*}
and
\begin{equation}
\mathrm{i}\partial_{t}q+\mathrm{D}q+\epsilon\mathrm{G}^*\mathcal{N}(\mathrm{G}q,\bar{\mathrm{G}}\bar{q})=0.
\end{equation}
The equation can ultimately be expressed as a system of nonlinearly coupled differential equations.
\begin{equation*}
\mathbf{i}\dot{q}_n=-d_nq_n-\epsilon\sum_{k,m_1,m_2,m_3}\mathrm{G}^*_{n,k}\mathrm{G}_{k,m_1}\bar{\mathrm{G}}_{k,m_2}\mathrm{G}_{k,m_3}q_{m_1}\bar{q}_{m_2}q_{m_3}.
\end{equation*}
Seeing that $\mathrm{G}$ is a real and unitary operator, one has
\begin{equation}
\bar{\mathrm{G}}=\mathrm{G}, \quad \mathrm{G}^*_{n,k}=\mathrm{G}_{k,n},
\end{equation}
and
\begin{equation}\label{eq2}
\mathbf{i}\dot{q}_n=-d_nq_n-\epsilon\sum_{k,m_1,m_2,m_3}\mathrm{G}_{k,n}\mathrm{G}_{k,m_1}\mathrm{G}_{k,m_2}\mathrm{G}_{k,m_3}q_{m_1}\bar{q}_{m_2}q_{m_3}.
\end{equation}
Take the complex conjugate of equation \eqref{eq2}, one has
\begin{equation}
\mathbf{i}\dot{q}_n=d_nq_n+\epsilon\sum_{k,m_1,m_2,m_3}\mathrm{G}_{k,n}\mathrm{G}_{k,m_1}\mathrm{G}_{k,m_2}\mathrm{G}_{k,m_3}\bar{q}_{m_1}q_{m_2}\bar{q}_{m_3}.
\end{equation}
Below, we demonstrate that equation \eqref{eq2} satisfies the Hamiltonian formalism.
\begin{thm}
Let \begin{equation}\label{hami}\mathcal{H}(q,\bar{q})=\sum_{m}d_m\bar{q}_mq_m+\frac{\epsilon}{2}\sum_{m,k,m_1,m_2,m_3}\mathrm{G}_{k,m}\mathrm{G}_{k,m_1}\mathrm{G}_{k,m_2}
\mathrm{G}_{k,m_3}q_{m_1}\bar{q}_{m_2}q_{m_3}\bar{q}_{m},\end{equation}
one has
\begin{eqnarray}
   \mathbf{i}\partial_{t}q_n&=& -\frac{\partial \mathcal{H}}{\partial \bar{q}_n}, \\
   \mathbf{i}\partial_{t}\bar{q}_n&=& \frac{\partial \mathcal{H}}{\partial q_n}.
 \end{eqnarray}
\end{thm}
\begin{proof}
The proof is essentially a straightforward calculation. For completeness, we provide some of the key steps.
\begin{equation}
\begin{split}
\frac{\partial \mathcal{H}}{\partial \bar{q}_n}=& d_nq_n+\frac{\epsilon}{2}\sum_{k,m_1,m_2,m_3}\mathrm{G}_{k,n}\mathrm{G}_{k,m_1}\mathrm{G}_{k,m_2}\mathrm{G}_{k,m_3}q_{m_1}\bar{q}_{m_2}q_{m_3}\\
&+\frac{\epsilon}{2}\sum_{m,k,m_1,m_3}\mathrm{G}_{k,m}\mathrm{G}_{k,m_1}\mathrm{G}_{k,n}\mathrm{G}_{k,m_3}q_{m_1}\bar{q}_{m}\bar{q}_{m_3}\\
=&d_nq_n+\epsilon\sum_{k,m_1,m_2,m_3}\mathrm{G}_{k,n}\mathrm{G}_{k,m_1}\mathrm{G}_{k,m_2}\mathrm{G}_{k,m_3}q_{m_1}\bar{q}_{m_2}\bar{q}_{m_3}\\
=&-\mathbf{i}\partial_{t}q_n.
\end{split}
\end{equation}
\end{proof}
Let $$\frac{1}{2}\sum_{k}\mathrm{G}_{k,m}\mathrm{G}_{k,m_1}\mathrm{G}_{k,m_2}
\mathrm{G}_{k,m_3}=\mathcal{G}_{(m,m_1,m_2,m_3)},$$
then the Hamiltonian function \eqref{hami} can be represented as
\begin{equation}\label{Hamilton3.11}
\mathcal{H}(q,\bar{q})=\sum_{m}d_m\bar{q}_mq_m+\epsilon\sum_{m,,m_1,m_2,m_3}\mathcal{G}_{(m,m_1,m_2,m_3)}q_{m_1}\bar{q}_{m_2}q_{m_3}\bar{q}_{m}.
\end{equation}
Set $m=\min\{m,m_1,m_2,m_3\}, \ m_3=\max\{m,m_1,m_2,m_3\},$
one has
\begin{equation}\label{calg}
\begin{split}
|\mathcal{G}_{(m,m_1,m_2,m_3)}|^{\alpha}&\leq 3\sum_{k}e^{-\frac{1}{8}|k-m|}e^{-\frac{1}{8}|k-m_1|}e^{-\frac{1}{8}|k-m_2|}e^{-\frac{1}{8}|k-m_3|}\\
&\leq 3e^{-\frac{1}{8}|m_3-m|}\sum_{k}e^{-\frac{1}{8}|k-m_1|}e^{-\frac{1}{8}|k-m_2|}\\
&\leq 24e^{-\frac{1}{8}|m_3-m|}.
\end{split}
\end{equation}

\section{Localized state for nonlinear disorder stark model}

\subsection{Functional setting}
\begin{defn}
For any $d, \rho>0$ and complex sequences $q=(q_n)_{n\in\mathbb{Z}_1}$ with $\mathbb{Z}_1\subset\mathbb{Z}$, we say $q\in l_{d,\rho}^1$ if and only if
$$\Vert q\Vert_{d,\rho}=\sum_{n\in\mathbb{Z}_1}|q_n|\langle n\rangle^de^{\rho|n|}<\infty,$$
where $\langle n\rangle=\sqrt{1+|n|^2}$.
\end{defn}
\begin{defn}
For any $r,s>0$, denote $D_{d,\rho}(r,s)$ by the complex neighborhood of $\mathbb{T}^b\times \{0\}\times\{0\}\times\{0\}$ in $\mathbb{T}^b\times \mathbb{R}^b\times l_{d,\rho}^1\times l_{d,\rho}^1$, i.e.,
$$D_{d,\rho}(r,s)=\{(\theta,I,q,\bar{q}): |{\rm Im}\theta|<r, |I|<s^2, \Vert q\Vert_{d,\rho}+ \Vert \bar{q}\Vert_{d,\rho}<s\},$$
where $|\cdot|$ denotes the $l^1$ norm of complex vectors.
\end{defn}

Consider a function $F(\theta, I,q,\bar{q};\xi):D_{d,\rho}(r,s)\times \mathcal{O}\rightarrow \mathbb{C}$ real analytic about the variables $(x, y,q,\bar{q})\in D=D_{d,\rho}(r,s)$ and $C^1$-smooth\footnote{In the whole of this paper, the derivatives with respect to the parameter $\xi \in \mathcal{O}$ are understood in the sense of Whitney. } in Whitney's sense about the parameter $\xi\in\mathcal{O}$, where $\mathcal{O}$ is a closed region in $\mathbb{R}^b$. The function $F$ can be expanded into a Taylor-Fourier series
$$F(\theta, I, q,\bar{q};\xi)=\sum_{\alpha,\beta\in \mathbb{N}^{\mathbb{Z}_1}}F_{\alpha\beta}(\theta, I;\xi)q^\alpha\bar{q}^\beta,$$
where 
$$F_{\alpha\beta}(\theta, I;\xi)=\sum_{k\in \mathbb{Z}^b,l\in\mathbb{N}^b}F_{kl\alpha\beta}(\xi)I^le^{{\bf i}\langle k,\theta\rangle}.
$$

\begin{defn}
Given any nonzero multi-index $(\alpha,\beta)=(\alpha_n,\beta_n)_{n\in\mathbb{Z}_1}\in \mathbb{N}^{\mathbb{Z}_1}\times \mathbb{N}^{\mathbb{Z}_1}$ with finitely many nonvanishing components, we define
\begin{align*}
&n_{\alpha\beta}^+=\max\{n\in\mathbb{Z}: (\alpha_n,\beta_n)\neq 0\}, \\
&n_{\alpha\beta}^-=\min\{n\in\mathbb{Z}: (\alpha_n,\beta_n)\neq 0\}, \\
&n_{\alpha\beta}^*=\max\{|n_{\alpha\beta}^+|,|n_{\alpha\beta}^-|\},\\
&|\alpha|=\sum_{n\in\mathbb{Z}_1}\alpha_n, |\beta|=\sum_{n\in\mathbb{Z}_1}\beta_n.
\end{align*}

\footnote{ if $n > n_{\alpha\beta}^+, (\alpha_n,\beta_n)=0$; if $n < n_{\alpha\beta}^+, (\alpha_n,\beta_n)=0$}In particular, for $|\alpha|=|\beta|=0$, we set $n_{\alpha\beta}^+=n_{\alpha\beta}^-=n_{\alpha\beta}^*=0$.
\end{defn}

\begin{defn}
Consider a real analytic function $F(\theta, I,q,\bar{q};\xi)$ defined on $D=D_{d,\rho}(r,s)$, $C^1$ dependent on a parameter $\xi\in\mathcal{O}$. Let
$$
|F_{kl\alpha\beta}|_{\mathcal{O}}=\sup_{\xi\in\mathcal{O}}\left(|F_{kl\alpha\beta}|+\left|\frac{\partial F_{kl\alpha\beta}}{\partial \xi}\right|\right).$$
and
$$\Vert F_{\alpha\beta}\Vert_{\mathcal{O}}=\sum_{k\in\mathbb{Z}^b, l\in\mathbb{N}^b} |F_{kl\alpha\beta}|_{\mathcal{O}}|I|^le^{|k| |{\rm Im}\theta|},$$
$$\Vert F\Vert_{\mathcal{O}}=\sum_{k\in\mathbb{Z}^b, l\in\mathbb{N}^b,\alpha,\beta\in \mathbb{N}^{\mathbb{Z}_1}} |F_{kl\alpha\beta}|_{\mathcal{O}}|I|^le^{|k||{\rm Im}\theta|}|q^\alpha||\bar{q}^\beta|.$$
Then define the weighted norm of $F$ as
$$\Vert F\Vert_{D,\mathcal{O}}=\sup_{D}\Vert F\Vert_{\mathcal{O}}.$$
\end{defn}

\begin{defn}
Consider a real analytic function $F(\theta, I,q,\bar{q};\xi)$ on $D=D_{d,\rho}(r,s)$, $C^1$ dependent on a parameter $\xi\in\mathcal{O}$. For the Hamiltonian vector field $X_F=(\partial_IF,-\partial_\theta F, (-{\bf i}\partial_{q_n}F)_{n\in\mathbb{Z}_1},(-{\bf i}\partial_{\bar{q}_n}F)_{n\in\mathbb{Z}_1})$ associated with $F$ on $D\times \mathcal{O}$, define its norm by
$$\Vert X_F\Vert_{D,\mathcal{O}}=\Vert \partial_IF\Vert_{D,\mathcal{O}}+\frac{1}{s^2}\Vert  \partial_\theta F\Vert_{D,\mathcal{O}}+\sup_D\frac{1}{s}\sum_{n\in\mathbb{Z}_1}(\Vert \partial_{q_n}F\Vert_{\mathcal{O}}+\Vert \partial_{\bar{q}_n}F\Vert_{\mathcal{O}})\langle n\rangle^de^{|n|\rho}.$$
\end{defn}

Given two real analytic functions $F$ and $G$, let $\{\cdot, \cdot\}$ denote Poisson bracket of such functions, i.e.,
$$\{ F,G\}=\langle \partial_IF, \partial_\theta G\rangle-\langle \partial_\theta F, \partial_I G\rangle+{\bf i}\sum_{n\in \mathbb{Z}_1}(\partial_{q_n}F\cdot \partial_{\bar{q}_n}F-\partial_{\bar{q}_n}F\cdot \partial_{q_n}F).$$

For any $d,\rho,r,s>0$, let $F,G$ be two real analytic functions on $D=D_{d,\rho}(r,s)$, $C_W^1$ dependent on a parameter $\xi\in\mathcal{O}$.
\begin{lem}
The norm $\Vert \cdot \Vert_{D,\mathcal{O}}$ has the Banach algebraic property, i.e.,
$$\Vert FG\Vert_{D,\mathcal{O}}\le \Vert F \Vert_{D,\mathcal{O}} \Vert G \Vert_{D,\mathcal{O}}.$$
\end{lem}
\begin{proof}
By simple computation, one gets
$$(FG)(\theta, I,q,\bar{q};\xi)=\sum_{k\in\mathbb{Z}^b, l\in\mathbb{N}^b,\alpha,\beta\in \mathbb{N}^{\mathbb{Z}_1}} (FG)_{kl\alpha\beta}(\xi)e^{{\bf i}\langle k,\theta\rangle}I^lq^\alpha\bar{q}^\beta$$
 with
$$(FG)_{kl\alpha\beta}(\xi)=\sum_{\substack{k^1+k^2=k, l^1+l^2=l\\\alpha^1+\alpha^2=\alpha, \beta^1+\beta^2=\beta}}F_{k^1l^1\alpha^1\beta^1}(\xi)G_{k^2l^2\alpha^2\beta^2}(\xi).$$
It follows that
\begin{align*}
\Vert FG\Vert_{D,\mathcal{O}}=&\sum_{\substack{k\in\mathbb{Z}^b, l\in\mathbb{N}^b\\\alpha,\beta\in \mathbb{N}^{\mathbb{Z}_1}}} |(FG)_{kl\alpha\beta}|_{\mathcal{O}}|I|^le^{|k||{\rm Im}\theta|}|q^\alpha||\bar{q}^\beta|\\
=&\sum_{\substack{k\in\mathbb{Z}^b, l\in\mathbb{N}^b\\\alpha,\beta\in \mathbb{N}^{\mathbb{Z}_1}}} \Bigg|\sum_{\substack{k^1+k^2=k, l^1+l^2=l\\\alpha^1+\alpha^2=\alpha, \beta^1+\beta^2=\beta}}F_{k^1l^1\alpha^1\beta^1}(\xi)G_{k^2l^2\alpha^2\beta^2}(\xi)\Bigg|_{\mathcal{O}}|I|^le^{|k||{\rm Im}\theta|}|q^\alpha||\bar{q}^\beta|\\
\le&  \Vert F \Vert_{D,\mathcal{O}} \Vert G \Vert_{D,\mathcal{O}}.
\end{align*}
\end{proof}

\begin{lem}[Generalized Cauchy Inequalities, \cite{GYZ2014}]\label{Cauchy}
The components of the Hamiltonian vector field $X_F$ satisfy: for any $0<r^\prime <r, 0<\rho^\prime <\rho$,
\begin{align*}
&\Vert \partial_\theta F\Vert_{D_{d,\rho}(r^\prime,s)}\le \frac{c}{r-r^\prime}\Vert F\Vert_{D},\\
& \Vert \partial_I F\Vert_{D_{d,\rho}(r,s/2)}\le \frac{c}{s^2}\Vert F\Vert_{D},\\
&\sup_{D_{d,\rho}(r,s/2)}\sum_{n\in\mathbb{Z}_1}(\Vert \partial_{q_n}F\Vert_{\mathcal{O}}+\Vert \partial_{\bar{q}_n}F\Vert_{\mathcal{O}})\langle n\rangle^de^{|n|\rho^\prime}\le \frac{c}{s(\rho-\rho^\prime)}\Vert F\Vert_{D}.
\end{align*}
\end{lem}

\begin{lem}[\cite{GYZ2014}]\label{czlem}
If $\Vert X_F\Vert_{D}<\varepsilon^\prime, \Vert X_G\Vert_{D}<\varepsilon^{\prime\prime}$, then
$$\Vert X_{\{F,G\}}\Vert_{D_{d,\rho}(r-\sigma,s-\delta)}<c\sigma^{-1}\delta^{-2}\varepsilon^\prime\varepsilon^{\prime\prime}$$
for any $0<\sigma<r$ and $0<\delta<s$.
\end{lem}

\begin{defn}
The function $F(\theta, I,q,\bar{q};\xi)$ is said to have gauge invariance if
$$F_{kl\alpha\beta}\equiv 0,\quad {\rm when}\ k_1+k_2+\cdots+k_b+|\alpha|-|\beta|\neq 0.$$
\end{defn}

\begin{lem}
If both $F$ and $G$ have gauge invariance, then $\{F,G\}$ has gauge invariance.
\end{lem}
\begin{proof}
Consider functions
$$F(\theta, I,q,\bar{q};\xi)=\sum_{k\in\mathbb{Z}^b,\alpha,\beta\in \mathbb{N}^{\mathbb{Z}_1}} F_{k\alpha\beta}(I;\xi)e^{{\bf i}\langle k,\theta\rangle}q^\alpha\bar{q}^\beta,$$
$$ G(\theta, I,q,\bar{q};\xi)=\sum_{k\in\mathbb{Z}^b, \alpha,\beta\in \mathbb{N}^{\mathbb{Z}_1}} G_{k\alpha\beta}(I;\xi)e^{{\bf i}\langle k,\theta\rangle}q^\alpha\bar{q}^\beta.$$
By gauge invariance, one gets
$$F_{k\alpha\beta}(I;\xi)=G_{k\alpha\beta}(I;\xi)=0,\quad {\rm when}\ k_1+k_2+\cdots+k_b+|\alpha|-|\beta|\neq 0.$$
By computation, we have
$$\{ F,G\}(\theta, I,q,\bar{q};\xi)=\sum_{k\in\mathbb{Z}^b,\alpha,\beta\in \mathbb{N}^{\mathbb{Z}_1}} \{F,G\}_{k\alpha\beta}(I;\xi)e^{{\bf i}\langle k,\theta\rangle}q^\alpha\bar{q}^\beta$$
with
\begin{align*}
 \{F,G\}_{k\alpha\beta}(I;\xi)=\sum_{\substack{k^1+k^2=k\\\alpha^1+\alpha^2=\alpha, \beta^1+\beta^2=\beta}}&\Big({\bf i} k^1 F_{k^1\alpha^1\beta^1}(I;\xi)\cdot \partial_I G_{k^2\alpha^2\beta^2}(I;\xi)\\
 &-\partial_I F_{k^1\alpha^1\beta^1}(I;\xi)\cdot {\bf i}k^2 G_{k^2\alpha^2\beta^2}(I;\xi)\Big)\\
 + \sum_{\substack{k^1+k^2=k\\\alpha^1+\alpha^2=\alpha, \beta^1+\beta^2=\beta}}\sum_n&\Big(F_{k^1(\alpha^1+e_n)\beta^1}(I;\xi) G_{k^2\alpha^2(\beta^2+e_n)}(I;\xi)\\
 &-F_{k^1\alpha^1(\beta^1+e_n)}(I;\xi) G_{k^2(\alpha^2+e_n)\beta^2}(I;\xi)\Big).
\end{align*}
In the following, we show $\{F,G\}$ has gauge invariance. Actually, if $\sum_{i=1}^b k_i+|\alpha|-|\beta|\neq 0$, then the followings can not happen
$$\sum_{i=1}^b k_i^1+|\alpha^1|-|\beta^1|=\sum_{i=1}^b k_i^2+|\alpha^2|-|\beta^2|=0,$$
$$\sum_{i=1}^b k_i^1+|\alpha^1+e_n|-|\beta^1|=\sum_{i=1}^b k_i^2+|\alpha^2|-|\beta^2+e_n|=0$$
or
$$\sum_{i=1}^b k_i^1+|\alpha^1|-|\beta^1+e_n|=\sum_{i=1}^b k_i^2+|\alpha^2+e_n|-|\beta^2|=0.$$
It implies that $ \{F,G\}_{k\alpha\beta}(I;\xi)=0$.
\end{proof}





\subsection{The abstract KAM theorem}\

Consider the Hamiltonian functions
\begin{equation}\label{2.1}
H=N+P=e(\xi)+\langle \omega(\xi),I\rangle+\sum_{n\in\mathbb{Z}_1}\Omega_n(\xi)q_n\bar{q}_n+P(\theta, I, q,\bar{q};\xi)
\end{equation}
defined on the domain $D=D_{d,\rho}(r,s)$, parametrized by $\xi \in \mathcal{O}=\left[-\frac{1}{10},\frac{1}{10}\right]^b$. The Hamiltonian is associated with the symplectic structure $\sum_{n\in \mathcal{J}}d I_n\wedge d\theta_n+{\bf i}\sum_{n\in \mathbb{Z}_1}dq_n\wedge d \bar{q}_n$. We want to show if the perturbation is small enough, the Hamiltonian $H$ possesses invariant tori for most $\xi\in \mathcal{O}$.

In the following, we introduce some conditions:

\noindent {\bf Assumption A} (Asymptotics of tangential frequencies): For each $n\in\mathcal{J}$, $\omega_n(\xi)=n+\omega^*_n(\xi)$ with $\left|\frac{\partial \omega^*_n}{\partial \xi_j}-\delta_{jn}\right|\leq  \frac{1}{4b}$.

\noindent {\bf Assumption B} (Asymptotics of normal frequencies): For each $n\in \mathbb{Z}_1$, $\Omega_n(\xi)=n+\Omega^*_n(\xi)$ with $|\Omega_n-\Omega_m|\geq \frac{2}{3} (m\neq n)$ and $\left|\frac{\partial \Omega^*_n}{\partial \xi_j}\right|\leq  \frac{1}{4b}$.

\noindent {\bf Assumption C} (Regularity of the perturbation): The perturbation $P$ is real analytic in $(\theta, I,q,\bar{q})\in D$ and $C^1$-smooth about the parameter $\xi\in\mathcal{O}$.

\noindent {\bf Assumption D} (Decay property of the perturbation): The perturbation $P$ can be written as follows
$$P=P^{low}+P^{high},$$
with
$$P^{low}:=\sum_{\substack{\alpha,\beta\in \mathbb{N}^{\mathbb{Z}_1}\\|\alpha+\beta|\le 2}}P^{low}_{\alpha\beta}(\theta,I;\xi)q^\alpha \bar{q}^\beta=\sum_{\substack{k\in \mathbb{Z}^b,l\in \mathbb{N}^b,\alpha,\beta\in \mathbb{N}^{\mathbb{Z}_1}\\2|l|+|\alpha+\beta|\le 2}}P_{kl\alpha\beta}I^l e^{{\bf i}\langle k,\theta\rangle }q^\alpha \bar{q}^\beta,$$
and
$$P^{high}=\sum_{\substack{k\in \mathbb{Z}^b,l\in \mathbb{N}^b,\alpha,\beta\in \mathbb{N}^{\mathbb{Z}_1}\\2|l|+|\alpha+\beta|\geq 3}}P_{kl\alpha\beta}I^l e^{{\bf i}\langle k,\theta\rangle }q^\alpha \bar{q}^\beta.$$
In addition, $P^{high}$ has the decomposition
$$P^{high}=\dot{P}+\ddot{P},$$
with
$$\dot{P}(q,\bar{q};\xi)=\sum_{\substack{\alpha,\beta\in \mathbb{N}^{\mathbb{Z}_1} \\|\alpha|=|\beta|=2}}\dot{P}_{\alpha\beta}(\xi)q^\alpha\bar{q}^\beta,$$
$$\ddot{P}(\theta, I, q,\bar{q};\xi)=\sum_{\alpha,\beta\in \mathbb{N}^{\mathbb{Z}_1}}\ddot{P}_{\alpha\beta}(\theta, I;\xi)q^\alpha\bar{q}^\beta.$$
The coefficients of $P$ satisfies
$$\Vert P^{low}_{\alpha\beta}\Vert_{D,\mathcal{O}}\le \varepsilon e^{-\rho n^*_{\alpha\beta}},$$
$$
\Vert \dot{P}_{\alpha\beta}\Vert_{D,\mathcal{O}}\le
e^{-\rho(n_{\alpha\beta}^+-n_{\alpha\beta}^-)},\quad\Vert \ddot{P}_{\alpha\beta}\Vert_{D,\mathcal{O}}\le
e^{-\rho n_{\alpha\beta}^*}.
$$

 \noindent {\bf Assumption E} (Gauge invariance of $P$): For
$$P=\sum_{\substack{k\in\mathbb{Z}^b, l\in\mathbb{N}^b,\alpha,\beta\in \mathbb{N}^{\mathbb{Z}_1}}}P_{kl\alpha\beta}I^le^{{\bf i}\langle k,\theta\rangle}q^\alpha\bar{q}^\beta,$$
 we have
 $$P_{kl\alpha\beta}\equiv 0, \quad {\rm if}\quad \sum_{j=1}^bk_j+|\alpha|-|\beta|\neq 0.$$

 \begin{thm}\label{main theorem}
 Suppose the Hamiltonian $H$ in \eqref{2.1} satisfies the assumptions {\bf Assumption A-E}. There exists a positive constant $\varepsilon_0=\varepsilon_0(r,s,d,\rho,\mathcal{J})$ such that if
 $$\Vert X_{P^{low}}\Vert_{D,\mathcal{O}}=\varepsilon<\varepsilon_0,\quad \Vert X_{P^{high}}\Vert_{D,\mathcal{O}}<1.$$
 then there exists a Cantor set $\mathcal{O}_\varepsilon\subset \mathcal{O}$ with $$ \mathrm{meas}(\mathcal{O}\setminus \mathcal{O}_\varepsilon)\leq c\varepsilon^{\frac{1}{16}}, $$  such that the following holds:

 (1) There is a $C^1$ map $\omega^*:\mathcal{O}_\varepsilon\rightarrow \mathbb{R}^b$ such that
 $$|\omega^*-\omega|_{\mathcal{O}_\varepsilon}\le c\varepsilon^{\frac{5}{6}}.$$

 (2)  There exists a map $\Phi: \mathbb{T}^b\times\mathcal{O}_\varepsilon\rightarrow D_{d,0}(\frac{r}{2},\frac{s}{2})$, real analytic in $\theta$ and $C^1$ parametrized by $\xi\in \mathcal{O}_\varepsilon$,
 such that
 $$\Vert\Phi-\Phi_0\Vert_{ D_{d,0}(\frac{r}{2},\frac{s}{2}),\mathcal{O}_\varepsilon}\le c\varepsilon^{\frac{5}{6}},$$
  where $\Phi^0$ is the
 trivial embedding $ \mathbb{T}^b\times\mathcal{O}_\varepsilon\rightarrow \mathbb{T}^b\times \{0\}\times\{0\}\times\{0\}$.

 (3) For any $\theta\in \mathbb{T}^b$ and $\xi\in \mathcal{O}_\varepsilon$, $\Phi(\theta+\omega^* t,\xi)=(\theta+\omega^* t,,I(t),q(t),\bar{q}(t))$ is a $b$-frequency quasi-periodic solution of equations of motion associated with \eqref{2.1}.

 (4) For each $t$, $q(t)=(q_n(t))_{n\in\mathbb{Z}}\in l_{d,0}^1(\mathbb{Z}).$
 \end{thm}

\begin{rem}\

$\bullet$ In this paper, gauge invariance in Assumption E serves to eliminate the presence of certain non-resonance conditions  (see some of the same-sign second Melnikov conditions) during the nonlinear KAM iteration. These non-resonance conditions fail to satisfy the twist property with respect to parameters, thereby obstructing the exclusion of the resonant set. In essence, this difficulty stems from the fact that only finitely many random variables are utilized as parameters; if all random variables could be employed as parameters, gauge invariance would no longer be necessary.

$\bullet$ Let ${\psi_n(x)}_{n \in \mathbb{Z}}$ be the a complete orthogonal basis of the linear operator $\mathcal{L}$(see equation \eqref{s}),
According to the fourth conclusion of Theorem \ref{main theorem}, the equation \ref{eq1.1} admits a solution of the form
 \begin{equation}
 u(t,x)=\sum_{n \in \mathbb{Z}}q_n(t)\psi_n(x)
 \end{equation}
 where
 $$\sup_{t}\sum_{n\in\mathbb{Z}}n^{2d}|q_n(t)|^2 <c\left(\sup_{t}\sum_{n\in\mathbb{Z}}n^{d}|q_n(t)|\right)^2<\infty.$$
 Applying the unitary transformation $\mathrm{G}$ from theorem \ref{yubei}, one can obtain the solution \eqref{solution} present in theorem \ref{thm1.1}.
\end{rem}

\subsection{The proof of Theorem \ref{thm1.1}}\

 Fix $\mathcal{J}=\{n_1,\cdots,n_b\}\subset \mathbb{Z}$ and $\mathbb{Z}_1=\mathbb{Z}\setminus \mathcal{J}$. Furthermore, we choose $n_i$ satisfying
$$|n_i|\le\frac{1}{6}| \ln \epsilon|, \quad i=1,\cdots,b.$$
 Given $\epsilon$ small enough, we introduce action-angular variables $(\theta, I)$ and initial data $y$ to Hamiltonian \eqref{Hamilton3.11} and let
 \begin{equation}\label{acangu}
q_n=\sqrt{I_n+y_n}e^{{\bf i}\theta_n},\quad \bar{q}_n=\sqrt{I_n+y_n}e^{-{\bf i}\theta_n},\quad n\in\mathcal{J}.
\end{equation}
Here $(I,\theta)=(I_{n_1},\cdots,I_{n_b},\theta_{n_1},\cdots,\theta_{n_b})$  are the standard action-angular variables in the $(q_n,\bar{q}_n)_{n\in\mathcal{J}}$-space and $(q,\bar{q})=(q_n,\bar{q}_n)_{n\in\mathbb{Z}_1}.$  Moreover, we take the parameters
$$\xi_{n_i}=v_{n_i},\quad {\rm for}\ 1\le i\le b$$
and $\mathcal{O}=\left[-\frac{1}{10},\frac{1}{10}\right]^b$.
Then the Hamiltonian \eqref{Hamilton3.11} becomes
\begin{equation}\label{huxia}
H(\theta, I, q,\bar{q};\xi)=N(\theta, I, q,\bar{q};\xi)+P(\theta, I, q,\bar{q};\xi)
\end{equation}
with
\begin{align*}
N(\theta, I, q,\bar{q};\xi)=&e(\xi)+\langle \omega(\xi), I\rangle+\sum_{j\in \mathbb{Z}_1} \Omega_j(\xi) |q_j|^2,
\end{align*}
where
$$e(\xi)=\sum_{n\in \mathcal{J}}d_n(\xi) y_n,\quad \langle \omega(\xi),I\rangle=\sum_{n\in \mathcal{J}} d_n(\xi) I_n,\quad  \Omega_j(\xi) =d_j(\xi),$$
\begin{align*}
P(\theta, I, q,\bar{q};\xi)=&P^{low}(\theta, I, q,\bar{q};\xi)+P^{high}(\theta, I, q,\bar{q};\xi),
\end{align*}
where
$$P^{low}(\theta, I, q,\bar{q};\xi)=\epsilon\sum_{\substack{{\rm at}\ {\rm least}\ {\rm two}\\ i,j,n,m\in\mathcal{J}}}\mathcal{G}_{(i,j,n,m)}(\xi)q_iq_j\bar{q}_n\bar{q}_m:=\sum_{\substack{\alpha,\beta\in \mathbb{N}^{\mathbb{Z}_1} \\|\alpha|+|\beta|\le 2}}P^{low}_{\alpha\beta}(\theta,I;\xi)q^\alpha\bar{q}^\beta$$
and
$$P^{high}(\theta, I, q,\bar{q};\xi)=\dot{P}(q,\bar{q};\xi)+\ddot{P}(\theta, I, q,\bar{q};\xi),$$
with
$$\dot{P}(q,\bar{q};\xi)=\epsilon\sum_{i,j,n,m\in\mathbb{Z}_1}\mathcal{G}_{(i,j,n,m)}(\xi)q_iq_j\bar{q}_n\bar{q}_m:=\sum_{\substack{\alpha,\beta\in \mathbb{N}^{\mathbb{Z}_1} \\|\alpha|=|\beta|=2}}\dot{P}_{\alpha\beta}(\xi)q^\alpha\bar{q}^\beta,$$
$$\ddot{P}(\theta, I, q,\bar{q};\xi)=\epsilon\sum_{\substack{{\rm only}\ {\rm one}\ i,j,n,m\in\mathcal{J}}}\mathcal{G}_{(i,j,n,m)}(\xi)q_iq_j\bar{q}_n\bar{q}_m:=\sum_{\substack{\alpha,\beta\in \mathbb{N}^{\mathbb{Z}_1} \\|\alpha|+|\beta|=3}}\ddot{P}_{\alpha\beta}(\theta,I;\xi)q^\alpha\bar{q}^\beta.$$

\begin{lem}\label{ver}
The function $P$ has gauge invariance.
\end{lem}
\begin{proof}
 For the function $\dot{P}$, since $k=0, |\alpha|=|\beta|=2$, it immediately follows $\sum_{i=1}^bk_i+|\alpha|-|\beta|=0$ holds, otherwise $\dot{P}_{\alpha\beta}(\xi)=0$.

For $P^{low}$, we rewrite
$$\sum_{\substack{{\rm at}\ {\rm least}\ {\rm two}\\ i,j,n,m\in\mathcal{J}}}\mathcal{G}_{(i,j,n,m)}(\xi)q_iq_j\bar{q}_n\bar{q}_m=\sum_{\substack{\alpha_{\mathcal{J}},\beta_{\mathcal{J}},\alpha,\beta\\ |\alpha_{\mathcal{J}}|+|\alpha|=|\beta_{\mathcal{J}}|+|\beta|}}R_{\alpha_{\mathcal{J}}\beta_{\mathcal{J}}\alpha\beta}(\xi)q_{\mathcal{J}}^{\alpha_{\mathcal{J}}}\bar{q}_{\mathcal{J}}^{\beta_{\mathcal{J}}}q^{\alpha}\bar{q}^{\beta},$$
where $\alpha_{\mathcal{J}}=(\alpha_n)_{n\in\mathcal{J}}, \beta_{\mathcal{J}}=(\beta_n)_{n\in\mathcal{J}}$, $q_{\mathcal{J}}=(q_n)_{n\in\mathcal{J}}$ and $\bar{q}_{\mathcal{J}}=(\bar{q}_n)_{n\in\mathcal{J}}$.  Substituting \eqref{acangu} into this function, one gets
\begin{align*}
&P^{low}_{\alpha\beta}(\theta,I;\xi)=\sum_{k}P^{low}_{k\alpha\beta}(I;\xi)e^{{\bf i}\langle k,\theta\rangle}\\
=&\sum_{\alpha_{\mathcal{J}},\beta_{\mathcal{J}}}R_{\alpha_{\mathcal{J}}\beta_{\mathcal{J}}\alpha\beta}(\xi)\prod_{n\in \mathcal{J}}(I_n+y_n)^{\alpha_n+\beta_n}e^{{\bf i}(\alpha_n-\beta_n) \theta_n}.
\end{align*}
It means that $\sum_{i=1}^b k_i=\sum_{n\in \mathcal{J}}(\alpha_n-\beta_n)=|\alpha_{\mathcal{J}}|-|\beta_{\mathcal{J}}|$. Then, we have
$$\sum_{i=1}^b k_i+|\alpha|-|\beta|=|\alpha_{\mathcal{J}}|-|\beta_{\mathcal{J}}|+|\alpha|-|\beta|=0.$$
It implies  gauge invariance for the function $P^{low}$. Similarly, we can show the function $\ddot{P}$ has gauge invariance.
\end{proof}

Then, we will deduce the estimates
\begin{align}\label{pb11}
\Vert P^{low}_{\alpha\beta}\Vert_{D,\mathcal{O}},\quad \Vert \ddot{P}_{\alpha\beta}\Vert_{D,\mathcal{O}}\le
\epsilon^{\frac{1}{2}}e^{-\frac{1}{8}n^*_{\alpha\beta}},
\end{align}
and
\begin{align}\label{pa11}
\Vert \dot{P}_{\alpha\beta}\Vert_{D,\mathcal{O}}\le
\epsilon^{\frac{1}{2}} e^{-\frac{1}{8}(n_{\alpha\beta}^+-n_{\alpha\beta}^-)}.
\end{align}
Actually, the bounds \eqref{pb11},\eqref{pa11} can be easily obtained from the decay property \eqref{calg}. If $i,j,n,m\in\mathbb{Z}_1$, then
$$\epsilon |\mathcal{G}_{(i,j,n,m)}(\xi)|_{\mathcal{O}}\le 24\epsilon e^{-\frac{1}{8}(\max\{i,j,n,m\}-\min\{i,j,n,m\})}.$$
It concludes the estimate for $\dot{P}_{\alpha\beta}$ with $|\alpha|=|\beta|=2$.
If there is at least one $i,j,n,m\in\mathcal{J}$, then
\begin{align*}
\epsilon |\mathcal{G}_{(i,j,n,m)}(\xi)|_{\mathcal{O}}\le& 24\epsilon e^{-\frac{1}{8}(\max\{i,j,n,m\}-\min\{i,j,n,m\})}\\
\le& 24\epsilon e^{-\frac{1}{8}\max\{i,j,n,m\}}e^{\frac{1}{48}|\ln\epsilon|}\\
\le & \epsilon^{\frac{2}{3}} e^{-\frac{1}{8}\max\{i,j,n,m\}}.
\end{align*}
It implies the estimate for $(P^{low}+\ddot{P})_{\alpha\beta}$.

We need to verify the Hamiltonian $H=N+P$ satisfies the assumptions {\bf Assumption A-E} of the KAM theorem, in which  {\bf Assumption C} is obviously satisfied. {\bf  Assumption E} has been verified by Lemma \ref{ver}. The assumption {\bf Assumption D} can be obtained from  \eqref{pb11} and \eqref{pa11}  with $\rho=\frac{1}{8}$.

Since $\omega_n(\xi)=d_n(\xi) (n\in\mathcal{J})$ and $\Omega_j(\xi)=d_j(\xi) (j\in \mathbb{Z}_1)$, then by Theorem \ref{yubei}, the conditions {\bf Assumption A} and {\bf Assumption B} are valid, by takeing $\delta \leq c(\frac{1}{b})$ for some small constant $c$. From estimates \eqref{pb11} and \eqref{pa11}, one can immediately obtains
$$\Vert X_P\Vert_{D_{d,\rho}(r,s),\mathcal{O}}\le \epsilon^{\frac{1}{3}}:=\varepsilon.$$

\section{Proof of Theorem \ref{main theorem}}

Set
$$\varepsilon_0=\varepsilon,\quad s_0=s,\quad r_0=r,$$
$$\rho_0=\rho,\quad  K_{0}=2|\ln \varepsilon_{0}|.$$
For $\nu=1,2,\cdots$, define the following sequences
\begin{align}
&\varepsilon_\nu=c\varepsilon_{\nu-1}^{\frac{5}{4}},
& \text{the size of perturbation}\\ &K_{\nu}=2|\ln \varepsilon_{\nu-1}|K_{\nu-1}, &\text{the length of truncation of lattice base}\\& \rho_\nu=K_\nu^{-1}, &\text{the exponential weighting of Hilbert space }\label{haoyong}\\
&s_{\nu}=s_0\left(1-\sum_{i=2}^{\nu+1}2^{-i}\right), &\text{the width of angle variable}\\& r_\nu=r_0\left(1-\sum_{i=2}^{\nu+1}2^{-i}\right),&\text{the width of action variable}\\
& \gamma_{\nu}=\varepsilon_\nu^{\frac{1}{16}},& \text{the measure of removed parameter}\\
& D_\nu=D_{\rho_\nu}(r_\nu,s_\nu).\label{haode}& \text{the domain of Hamilton function}
\end{align}

\begin{lem}[Iterative lemma]\label{itlem}
There exists $\varepsilon_0$ sufficiently small such that the following holds:

$\mathbf{(I1)_\nu}:$ For any $\nu\geq 0$, there exist Hamiltonians
$$H_{\nu}(\theta, I, q,\bar{q};\xi)=N_\nu(I, q,\bar{q};\xi)+P_\nu(\theta, I, q,\bar{q};\xi),$$
with
$$N_\nu(I, q,\bar{q};\xi)=e_\nu(\xi)+\langle \omega^\nu(\xi), I\rangle+\sum_{n\in\mathbb{Z}_1}\Omega^\nu_n(\xi) q_n\bar{q}_n,$$
$$P_\nu=P_\nu^{low}+P_\nu^{high},$$
with
$$P_\nu^{low}=\sum_{\substack{k\in \mathbb{Z}^b,l\in \mathbb{N}^b,\alpha,\beta\in \mathbb{N}^{\mathbb{Z}_1}\\2|l|+|\alpha+\beta|\le 2}}P^\nu_{kl\alpha\beta}I^l e^{{\bf i}\langle k,\theta\rangle }q^\alpha \bar{q}^\beta:=\sum_{\substack{\alpha,\beta\in \mathbb{N}^{\mathbb{Z}_1}\\|\alpha+\beta|\le 2}}P^{low}_{\nu,\alpha\beta}(\theta,I;\xi)q^\alpha \bar{q}^\beta,$$
and
$$P_\nu^{high}=\sum_{\substack{k\in \mathbb{Z}^b,l\in \mathbb{N}^b,\alpha,\beta\in \mathbb{N}^{\mathbb{Z}_1}\\2|l|+|\alpha+\beta|\geq 3}}P^\nu_{kl\alpha\beta}I^l e^{{\bf i}\langle k,\theta\rangle }q^\alpha \bar{q}^\beta.$$
In addition, $P_\nu^{high}$ has the decomposition
$$P_\nu^{high}=\dot{P}+\ddot{P}_\nu,$$
with
$$\dot{P}(q,\bar{q};\xi)=\sum_{\substack{\alpha,\beta\in \mathbb{N}^{\mathbb{Z}_1} \\|\alpha|=|\beta|=2}}\dot{P}_{\alpha\beta}(\xi)q^\alpha\bar{q}^\beta,$$
$$\ddot{P}_\nu(\theta, I, q,\bar{q};\xi)=\sum_{\alpha,\beta\in \mathbb{N}^{\mathbb{Z}_1}}\ddot{P}^\nu_{\alpha\beta}(\theta, I;\xi)q^\alpha\bar{q}^\beta.$$

The Hamiltonian $H_\nu$ has the properties:

(1) $\omega^\nu$ and $\Omega^\nu$ satisfies
\begin{align}
&\omega_n^0=n+\omega^*_n(\xi),\quad |\omega_n^{\nu+1}-\omega_n^\nu|_{\mathcal{O}_{\nu+1}}\le \varepsilon_\nu^{\frac{5}{6}},  &(n\in \mathcal{J}),\\
&\Omega_n^0=n+\Omega^*_n(\xi), \quad |\Omega_n^{\nu+1}-\Omega^\nu_n|_{\mathcal{O}_{\nu+1}}\le \varepsilon_\nu^{\frac{5}{6}} e^{-\rho_{\nu}|n|},&(n\in \mathbb{Z}_1). \label{5.9}
\end{align}

(2) The Hamiltonian $H_{\nu}$ is real analytic on $D_\nu$ and $C^1$ parameterized by $\xi\in \mathcal{O}_\nu$, where
$$\mathcal{O}_0=\mathcal{O},$$
and for $\nu\geq 1$
$$\mathcal{O}_\nu=\left\{
\xi\in \mathcal{O}_{\nu-1}:
\begin{array}{l}
|\langle k,\omega_{\nu-1}\rangle|>\frac{\gamma_{\nu-1}}{|k|},\\
|\langle k,\omega_{\nu-1}\rangle+\Omega^{\nu-1}_n|>\frac{\gamma_{\nu-1}}{|k|^\tau K_\nu},\\
|\langle k,\omega_{\nu-1}\rangle+\Omega^{\nu-1}_n\pm \Omega^{\nu-1}_m|>\frac{\gamma_{\nu-1}}{|k|^\tau K_\nu^{2}},
\end{array} k\neq 0,|m|,|n|\le K_\nu
\right\}.$$

(3) The perturbation $P_\nu$ has gauge invariance and satisfies the estimate
$$\Vert X_{P^{low}_\nu}\Vert_{D_\nu,\mathcal{O}_{\nu}}\le  \varepsilon_\nu,\quad \Vert X_{P^{high}_\nu}\Vert_{D_\nu,\mathcal{O}_{\nu}}\le 1+\sum_{i=1}^\nu c\varepsilon^{\frac{5}{6}}_i\le 2$$
and for the coefficients of the perturbation $P_\nu$, the following holds:
\begin{align}\label{pb}
\Vert P^{low}_{\nu,\alpha\beta}\Vert_{D_\nu,\mathcal{O}_\nu}\le
\varepsilon_\nu e^{-\rho_\nu n^*_{\alpha\beta}} ,\quad \Vert \dot{P}_{\alpha\beta}\Vert_{D_\nu,\mathcal{O}_\nu}\le
e^{-\rho_\nu(n_{\alpha\beta}^+-n_{\alpha\beta}^-)},\quad
\end{align}
\begin{align}\label{pa}
 \Vert \ddot{P}^\nu_{\alpha\beta}\Vert_{D_\nu,\mathcal{O}_\nu}\le (1+\sum_{i=1}^\nu c\varepsilon_i^{\frac{5}{6}})
e^{-\rho_\nu n^*_{\alpha\beta}}\le 2e^{-\rho_\nu n^*_{\alpha\beta}}.
\end{align}

$\mathbf{(I2)_\nu}:$ For any $\nu\geq 0$, there exists symplectic transformation $\Phi_{\nu+1}: D_{\nu+1}\times \mathcal{O}_{\nu+1}\rightarrow D_{\nu}\times \mathcal{O}_{\nu}$ such that
$$H_{\nu}\circ\Phi_{\nu+1}=H_{\nu+1}$$
and $\Phi_\nu$ satisfies the estimate
$$\Vert \Phi_{\nu+1}-id\Vert_{D_{\nu+1},\mathcal{O}_{\nu+1}}\le c\varepsilon_\nu^{\frac{5}{6}},\quad \Vert D\Phi_{\nu+1}-Id\Vert_{D_{\nu+1},\mathcal{O}_{\nu+1}}\le c\varepsilon_\nu^{\frac{4}{5}}.$$
\end{lem}

In the following, we will describe one step of KAM iteration in more details. By defining $N_0=N$ and $P_0=P$, the corresponding properties for $H_0=N_0+P_0$ in Lemma \ref{itlem} holds naturally. Now, we suppose the iteration holds at the  $\nu$-th step and will show the iteration holds at  the $(\nu+1)$-th step.

To simplify the notations, we use the notations without subscripts (or superscripts) ``$\nu$" to represent the corresponding quantities at the $\nu$-th step;  the notations with subscripts (or superscripts) ``$+$" to represent the corresponding quantities at the $(\nu+1)$-th step. 

Thus, we consider the Hamiltonian
$$H(\theta, I, q,\bar{q};\xi)=N(I, q,\bar{q};\xi)+P(\theta, I, q,\bar{q};\xi),$$
with
\begin{align*}
N(I, q,\bar{q};\xi)=&e(\xi)+\langle \omega(\xi), I\rangle+\sum_{n\in\mathbb{Z}_1}\Omega_n(\xi) q_n\bar{q}_n,\\
P(\theta, I, q,\bar{q};\xi)=&P^{low}(\theta, I, q,\bar{q};\xi)+P^{high}(\theta, I, q,\bar{q};\xi).
\end{align*}

\subsection{The derivation of homological equation}\label{sub5.1}
By the definition of $P^{low}$ and $P^{high}$, we have
$$P^{low}=\sum_{\substack{k\in \mathbb{Z}^b,l\in \mathbb{N}^b,\alpha,\beta\in \mathbb{N}^{\mathbb{Z}_1}\\2|l|+|\alpha+\beta|\le 2}}P_{kl\alpha\beta}I^l e^{{\bf i}\langle k,\theta\rangle }q^\alpha \bar{q}^\beta,$$
and
$$P^{high}=\sum_{\substack{k\in \mathbb{Z}^b,l\in \mathbb{N}^b,\alpha,\beta\in \mathbb{N}^{\mathbb{Z}_1}\\2|l|+|\alpha+\beta|\geq 3}}P_{kl\alpha\beta}I^l e^{{\bf i}\langle k,\theta\rangle }q^\alpha \bar{q}^\beta.$$
We will do a symplectic coordinate change $\Phi$ which can be produced by the time-$1$ map $X_F^t|_{t=1}$ of the Hamiltonian vector field, where $F$ is of the form
$$F=\sum_{\substack{k\in \mathbb{Z}^b,l\in \mathbb{N}^b,\alpha,\beta\in \mathbb{N}^{\mathbb{Z}_1}\\2|l|+|\alpha|+|\beta|\le 2}}F_{kl\alpha\beta}I^lq^\alpha\bar{q}^\beta e^{{\bf i}\langle k,\theta\rangle}.$$
Under the transformation $\Phi=X_F^t|_{t=1}$, by Taylor's formula, we have
\begin{align*}
H_+=&H\circ \Phi=H\circ X_F^1\\
=&H+\{H,F\}+\int_0^1(1-t)\{\{H,F\},F\}\circ X_F^tdt\\
=& N+\{N,F\}+\int_0^1(1-t)\{\{N,F\},F\}\circ X_F^t dt\\
&\quad +P^{low}+\int_0^1 \{P^{low},F\}\circ X_F^tdt\\
&\quad +P^{high}+\{P^{high},F\}+\int_0^1 (1-t)\{\{P^{high},F\},F\}\circ X_F^tdt.
\end{align*}
The following is the homological equation
\begin{equation}\label{huhomo}
\{N,F\}+P^{low}+\{P^{high},F\}^{low}=\hat{N}+\hat{P},
\end{equation}
where $\hat{N}$ and $\hat{P}$ will be determined later.
If the above equation is solved, then the new normal form $N_+$ and new perturbation $P_+$ can be written as
\begin{align*}
N_+=&N+\hat{N},\\
P_+=&\hat{P}+P^{high}+\{P^{high},F\}^{high}\\
&\quad +\int_0^1 (1-t)\{\{N+P^{high},F\},F\}\circ X_F^tdt\\
&\quad +\int_0^1 \{P^{low},F\}\circ X_F^tdt.
\end{align*}

\subsection{The solvability of homological equation \eqref{huhomo}}\label{gugu}
Write $P^{low}$ as follows
$$P^{low}=P^\theta+P^I+P^1+P^2,$$
where
\begin{align}
P^\theta=&P^\theta(\theta;\xi)=\sum_{\substack{k\in \mathbb{Z}^b,l\in \mathbb{N}^b,\alpha,\beta\in \mathbb{N}^{\mathbb{Z}_1}\\2|l|+|\alpha+\beta|=0}}P_{kl\alpha\beta}I^l e^{{\bf i}\langle k,\theta\rangle }q^\alpha \bar{q}^\beta,\label{1pp}\\
P^I=&\langle P^I(\theta;\xi),I\rangle=\sum_{\substack{k\in \mathbb{Z}^b,l\in \mathbb{N}^b,\alpha,\beta\in \mathbb{N}^{\mathbb{Z}_1}\\|l|=1,|\alpha+\beta|=0}}P_{kl\alpha\beta}I^l e^{{\bf i}\langle k,\theta\rangle }q^\alpha \bar{q}^\beta,\nonumber\\
P^1=&\langle P^q(\theta;\xi),q\rangle+\langle P^{\bar{q}}(\theta;\xi),\bar{q}\rangle=\sum_{\substack{k\in \mathbb{Z}^b,l\in \mathbb{N}^b,\alpha,\beta\in \mathbb{N}^{\mathbb{Z}_1}\\|l|=0,|\alpha+\beta|=1}}P_{kl\alpha\beta}I^l e^{{\bf i}\langle k,\theta\rangle }q^\alpha \bar{q}^\beta,\nonumber\\
P^2=&\langle P^{qq}(\theta;\xi)q,q\rangle+\langle P^{q\bar{q}}(\theta;\xi)q,\bar{q}\rangle+\langle P^{\bar{q}\bar{q}}(\theta;\xi)\bar{q},\bar{q}\rangle\nonumber\\
=&\sum_{\substack{k\in \mathbb{Z}^b,l\in \mathbb{N}^b,\alpha,\beta\in \mathbb{N}^{\mathbb{Z}_1}\\|l|=0,|\alpha+\beta|=2}}P_{kl\alpha\beta}I^l e^{{\bf i}\langle k,\theta\rangle }q^\alpha \bar{q}^\beta,\label{2pp}
\end{align}
with $P^I=(P^{I_j})_{j\in \mathcal{J}}$, $P^q=(P^{q_n})_{n\in\mathbb{Z}_1}$ and   $P^{\bar{q}}=(P^{\bar{q}_n})_{n\in\mathbb{Z}_1}$ being vectors and $P^{qq}=(P^{q_mq_n})_{m,n\in\mathbb{Z}_1}$, $P^{q\bar{q}}=(P^{q_m\bar{q}_n})_{m,n\in\mathbb{Z}_1}$ and $P^{\bar{q}\bar{q}}=(P^{\bar{q}_m\bar{q}_n})_{m,n\in\mathbb{Z}_1}$ being matrices.

In addition, write $F=F^\theta+F^I+F^1+F^2$ as the same form as $P^{low}$ and $P^{high}=\sum_{j=0}^4 P^{(j)}$, where
\begin{align*}
P^{(0)}=&\sum_{\substack{k\in \mathbb{Z}^b,l\in \mathbb{N}^b,\alpha,\beta\in \mathbb{N}^{\mathbb{Z}_1}\\|l|=2, |\alpha+\beta|=0}}P_{kl\alpha\beta}I^l e^{{\bf i}\langle k,\theta\rangle }q^\alpha \bar{q}^\beta,\\
P^{(1)}=&\sum_{\substack{k\in \mathbb{Z}^b,l\in \mathbb{N}^b,\alpha,\beta\in \mathbb{N}^{\mathbb{Z}_1}\\|l|=1, |\alpha+\beta|=1}}P_{kl\alpha\beta}I^l e^{{\bf i}\langle k,\theta\rangle }q^\alpha \bar{q}^\beta,\\
P^{(2)}=&\sum_{\substack{k\in \mathbb{Z}^b,l\in \mathbb{N}^b,\alpha,\beta\in \mathbb{N}^{\mathbb{Z}_1}\\|l|=1, |\alpha+\beta|=2}}P_{kl\alpha\beta}I^l e^{{\bf i}\langle k,\theta\rangle }q^\alpha \bar{q}^\beta,\\
P^{(3)}=&\sum_{\substack{k\in \mathbb{Z}^b,l\in \mathbb{N}^b,\alpha,\beta\in \mathbb{N}^{\mathbb{Z}_1}\\|l|=0, |\alpha+\beta|=3}}P_{kl\alpha\beta}I^l e^{{\bf i}\langle k,\theta\rangle }q^\alpha \bar{q}^\beta,\\
P^{(4)}=&\sum_{\substack{k\in \mathbb{Z}^b,l\in \mathbb{N}^b,\alpha,\beta\in \mathbb{N}^{\mathbb{Z}_1}\\2|l|+|\alpha+\beta|\geq 5\ {\rm or}\ |\alpha+\beta|\geq 4}}P_{kl\alpha\beta}I^l e^{{\bf i}\langle k,\theta\rangle }q^\alpha \bar{q}^\beta.
\end{align*}
By simple computation, we obtain
$$\{P^{high}, F\}^{low}=\{P^{high}, F\}^I+\{P^{high}, F\}^1+\{P^{high}, F\}^2,$$
where
\begin{align*}
\{P^{high}, F\}^I=&\langle \partial_{I}P^{(0)}, \partial_{\theta} F^\theta\rangle+{\bf i}(\langle \partial_{\bar{q}}P^{(1)}, \partial_{q} F^1\rangle-\langle \partial_{q}P^{(1)}, \partial_{\bar{q}} F^1\rangle),\\
\{P^{high}, F\}^1=&\langle \partial_{I}P^{(1)}, \partial_{\theta} F^\theta\rangle,\\
\{P^{high}, F\}^2=&\langle \partial_{I}P^{(1)}, \partial_{\theta} F^1\rangle+ \langle \partial_{I}P^{(2)}, \partial_{\theta} F^\theta\rangle\\
&+{\bf i}(\langle \partial_{\bar{q}}P^{(1)}, \partial_{q} F^1\rangle-\langle \partial_{q}P^{(1)}, \partial_{\bar{q}} F^1\rangle).
\end{align*}

Set
\begin{equation}\label{ohat}
\hat{N}^\theta=P_0^\theta,\quad \hat{N}^{I_j}=P^{I_j}_0+\{P^{high}, F^{\theta}+F^1\}^{I_j}_0,
\end{equation}
\begin{equation}\label{Ohat}
\hat{N}^{q_m\bar{q}_n}=P^{q_m\bar{q}_n}_0+\{P^{high}, F^{\theta}+F^1\}^{q_m\bar{q}_n}_0,
\end{equation}
where  for a function $W(\theta)$, $W_0$ denotes its $0$-th Fourier coefficient. In addition, we define
\begin{align}\label{nN}
\hat{N}=&\hat{N}^\theta+\sum_{j\in\mathcal{J}}\hat{N}^{I_j} I_j+\sum_{n\in\mathbb{Z}_1}\hat{N}^{q_n\bar{q}_n}q_n\bar{q}_n\nonumber\\
:=&\hat{e}+\sum_{j\in\mathcal{J}}\hat{\omega}_j I_j+\sum_{n\in\mathbb{Z}_1}\hat{\Omega}_nq_n\bar{q}_n,
\end{align}
and
\begin{align}\label{oOme}
\hat{P}=&\sum_{|n|>K_+}[(P^{q_n}+\{P^{high}, F^{\theta}\}^{q_n})q_n+(P^{\bar{q}_n}+\{P^{high}, F^{\theta}\}^{\bar{q}_n})\bar{q}_n]\nonumber\\
&+\sum_{|m| {\rm or} |n|>K_+}[(P^{q_mq_n}+\{P^{high}, F^{\theta}+F^1\}^{q_mq_n})q_mq_n\nonumber\\
&\quad\quad\quad\quad\quad\quad+(P^{\bar{q}_m\bar{q}_n}+\{P^{high}, F^{\theta}+F^1\}^{\bar{q}_m\bar{q}_n})\bar{q}_m\bar{q}_n\nonumber\\
&\quad\quad\quad\quad\quad\quad+(P^{q_m\bar{q}_n}+\{P^{high}, F^{\theta}+F^1\}^{q_m\bar{q}_n})q_m\bar{q}_n].
\end{align}

Denote $\partial_\omega=\omega\cdot \partial_\theta$. Then homological equation \eqref{huhomo} becomes
\begin{align}
&\partial_\omega F^\theta+\hat{N}^\theta=P^\theta,\label{1ho}\\
&(\partial_\omega +{\bf i} \Omega_n)F^{q_n}=P^{q_n}+\{P^{high}, F^{\theta}\}^{q_n},\quad |n|\le K_+,\label{2ho}\\
&(\partial_\omega -{\bf i} \Omega_n)F^{\bar{q}_n}=P^{\bar{q}_n}+\{P^{high}, F^{\theta}\}^{\bar{q}_n},\quad |n|\le K_+,\label{3ho}\\
&\partial_\omega F^{I_j}+\hat{N}^{I_j}=P^{I_j}+\{P^{high}, F^{\theta}+F^1\}^{I_j}, \quad j\in\mathcal{J},\label{4ho}\\
&(\partial_\omega +{\bf i} \Omega_m+{\bf i} \Omega_n)F^{q_mq_n}=P^{q_mq_n}+\{P^{high}, F^{\theta}+F^1\}^{q_mq_n},\quad |m|,|n|\le K_+,\label{5ho}\\
&(\partial_\omega -{\bf i} \Omega_m-{\bf i} \Omega_n)F^{\bar{q}_m\bar{q}_n}=P^{\bar{q}_m\bar{q}_n}+\{P^{high}, F^{\theta}+F^1\}^{\bar{q}_m\bar{q}_n},\quad |m|,|n|\le K_+,\label{6ho}\\
&(\partial_\omega +{\bf i} \Omega_m-{\bf i} \Omega_n)F^{q_m\bar{q}_n}+\delta_{mn}\hat{N}^{q_m\bar{q}_n}\nonumber\\
&\quad\quad\quad\quad\quad\quad\quad\quad\quad\quad\ =P^{q_m\bar{q}_n}+\{P^{high}, F^{\theta}+F^1\}^{q_m\bar{q}_n},\quad |m|,|n|\le K_+.\label{7ho}
\end{align}

\subsection{The solution of homological equation}
 In this subsection, we will focus on the solution of homological equation. Define the domain
$$D_i=D_{d, \rho}\left(r_++\frac{i}{5}(r-r_+),s_++ \frac{i}{5}(s-s_+)\right),\quad 1\le i\le 5.$$
It immediately follows that
$$D_1\subset D_2\subset D_3\subset D_4\subset D_5.$$
Then, we have the following lemma.
\begin{lem}\label{holemma}
There exists a real analytic Hamiltonian
\begin{align*}
F=&\sum_{\substack{k\in \mathbb{Z}^b,l\in \mathbb{N}^b,\alpha,\beta\in \mathbb{N}^{\mathbb{Z}_1}\\2|l|+|\alpha|+|\beta|\le 2}}F_{kl\alpha\beta}I^lq^\alpha\bar{q}^\beta e^{{\bf i}\langle k,\theta\rangle}\\
=&F^\theta+F^I+F^1+F^2,
\end{align*}
where $F^\theta,F^I,F^1,F^2$ is defined similar as $P^\theta,P^I,P^1,P^2$ in \eqref{1pp}-\eqref{2pp}. In addition, $F$ is $C^1$ parametrized by $\xi\in \mathcal{O}_+$ with
$$\mathcal{O}_+=\left\{
\xi\in \mathcal{O}:
\begin{array}{l}
|\langle k,\omega\rangle|>\frac{\gamma}{|k|^\tau},\\
|\langle k,\omega\rangle+\Omega_n|>\frac{\gamma}{|k|^\tau K_+},\\
|\langle k,\omega\rangle+\Omega_n\pm \Omega_m|>\frac{\gamma}{|k|^\tau K_+^{2}},
\end{array} \quad k\neq 0,|m|,|n|\le K_+
\right\}.$$
 such that $F$ satisfies the following equation
$$\{N,F\}+P^{low}+\{P^{high},F\}^{low}=\hat{N}+\hat{P}.$$
where $\hat{N}$ and $\hat{P}$ are defined by \eqref{nN} and \eqref{oOme}, respectively. Moreover, the function $F$ has gauge invariance, and for $\varepsilon$ sufficiently small, the followings hold
\begin{align}
&\Vert F^\theta\Vert_{D_4,\mathcal{O}_+} \le c\gamma^{-2} \varepsilon (r-r_+)^{-(2\tau+b+1)},\label{1est}\\
&\Vert F^{q_n}\Vert_{D_3,\mathcal{O}_+},\Vert F^{\bar{q}_n}\Vert_{D_3,\mathcal{O}_+} \le c\gamma^{-4} K_+^2  \varepsilon (r-r_+)^{-(4\tau+2b+2)}e^{-\rho |n|}, \quad |n|\le K_+,\label{2est}\\
&\Vert F^{I_j}\Vert_{D_2,\mathcal{O}_+} \le c\gamma^{-6} K_+^2  \varepsilon (r-r_+)^{-(6\tau+3b+3)},\quad j\in\mathcal{J},\label{3est}\\
&\Vert F^{q_mq_n}\Vert_{D_2,\mathcal{O}_+},\Vert F^{q_m\bar{q}_n}\Vert_{D_2,\mathcal{O}_+} ,\Vert F^{\bar{q}_m\bar{q}_n}\Vert_{D_2,\mathcal{O}_+}\nonumber \\
&\quad\quad\le c\gamma^{-6} K_+^{6}  \varepsilon (r-r_+)^{-(6\tau+3b+3)}e^{-\rho \max\{|m|,|n|\}}, \quad |m|,|n|\le K_+.\label{4est}
\end{align}
\end{lem}

\begin{proof}
From the analysis of subsection \ref{gugu}, we know homological equation are reduced to equations \eqref{1ho}-\eqref{7ho}.  Note that these equations can be solved one by one. The estimate \eqref{1est}-\eqref{4est} can be obtained following the order
$$\eqref{1est}\Rightarrow \eqref{2est}\Rightarrow \eqref{3est},\eqref{4est}.$$

Since $P$ has gauge invariance, by the definition of $P^{low}$ and $P^{high}$, we know both the functions $P^{low}$ and $P^{high}$ have gauge invariance. By the gauge invariance of $P^{low}$, we can obtain the gauge invariance of $P^\theta$. Then, $F^\theta$ has gauge invariance according to equation \eqref{1ho}. Since gauge invariance are maintained under Poisson bracket, we obtain $\{P^{high}, F^{\theta}\}$ has gauge invariance. Therefore, one has
\begin{equation}\label{gau1}
P_0^{q_n}+\{P^{high}, F^{\theta}\}_0^{q_n}=0,\quad P^{\bar{q}_n}_0+\{P^{high}, F^{\theta}\}^{\bar{q}_n}_0=0.
\end{equation}
Then, $F^1$ has gauge invariance according to  \eqref{2ho} and \eqref{3ho}. As a consequence, we can deduce the gauge invariance for $\{P^{high}, F^{\theta}+F^1\}$. It follows that
\begin{equation}\label{gau2}
P^{q_mq_n}_0+\{P^{high}, F^{\theta}+F^1\}^{q_mq_n}_0=0,\quad P^{\bar{q}_m\bar{q}_n}_0+\{P^{high}, F^{\theta}+F^1\}^{\bar{q}_m\bar{q}_n}_0=0.
\end{equation}

Expanding the functions in equations \eqref{1ho}-\eqref{7ho} into Fourier series, from \eqref{gau1} and \eqref{gau2}, equations \eqref{1ho}-\eqref{7ho} are reduced into the following equations:
for $k\neq 0$ and $|m|,|n|\le K_+$,
\begin{align*}
&\langle k,\omega\rangle F^\theta_{k}={\bf i}P^\theta_{k},\\
&(\langle k,\omega \rangle +\Omega_n)F^{q_n}_{k}={\bf i} P^{q_n}_{k}+{\bf i}\{P^{high}, F^{\theta}\}_k^{q_n},\\
&(\langle k,\omega \rangle -\Omega_n)F^{\bar{q}_n}_{k}={\bf i} P^{\bar{q}_n}_{k}+{\bf i}\{P^{high}, F^{\theta}\}_k^{\bar{q}_n},\\
&\langle k,\omega\rangle F_k^{I_j}={\bf i}P_k^{I_j}+{\bf i}\{P^{high}, F^{\theta}+F^1\}_k^{I_j}, \\
&(\langle k,\omega\rangle+\Omega_m+\Omega_n)F_k^{q_mq_n}={\bf i}P_k^{q_mq_n}+{\bf i}\{P^{high}, F^{\theta}+F^1\}_k^{q_mq_n},\\
&(\langle k,\omega\rangle-\Omega_m-\Omega_n)F_k^{\bar{q}_m\bar{q}_n}={\bf i}P_k^{\bar{q}_m\bar{q}_n}+{\bf i}\{P^{high}, F^{\theta}+F^1\}_k^{\bar{q}_m\bar{q}_n},
\end{align*}
and for all $k\in\mathbb{Z}^b$, $|m|,|n|\le K_+$ and $|k|+|m-n|\neq 0$,
\begin{equation}\label{huaxue}
(\langle k,\omega\rangle+\Omega_m-\Omega_n)F_k^{q_m\bar{q}_n}={\bf i} P_k^{q_m\bar{q}_n}+{\bf i} \{P^{high}, F^{\theta}+F^1\}_k^{q_m\bar{q}_n}.
\end{equation}

We will consider equation \eqref{huaxue}, the other equations can be solved analogous. Since $\xi\in \mathcal{O}_+$, we have for $k\neq 0$,
\begin{equation}\label{fmn}
F^{q_m\bar{q}_n}_k=\frac{{\bf i} P_k^{q_m\bar{q}_n}+{\bf i} \{P^{high}, F^{\theta}+F^1\}_k^{q_m\bar{q}_n}}{\langle k,\omega \rangle +\Omega_m-\Omega_n}.
\end{equation}
By \eqref{pb}, \eqref{pa}, \eqref{1est} and \eqref{2est}, we have
\begin{equation}\label{haoy1}
|P_k^{q_m\bar{q}_n}|_{\mathcal{O}}\le  \varepsilon e^{-\rho \max\{|m|,|n|\}}e^{-|k|r_3},
\end{equation}
\begin{equation}\label{haoy2}
|\{P^{high}, F^{\theta}+F^1\}_k^{q_m\bar{q}_n}|_{\mathcal{O}_+}\le  c\gamma^{-4} K_+^2  \varepsilon (r-r_+)^{-(4\tau+2b+2)} e^{-\rho\max\{|m|, |n|\}}e^{-|k|r_3},
\end{equation}
where $r_i=r_++\frac{i}{5}(r-r_+)$ for $1\le i\le 5$.
Therefore, we obtain
\begin{equation}\label{1ff}
\sup_{\xi\in\mathcal{O}_+}|F^{q_m\bar{q}_n}_k|\le c \gamma^{-5}|k|^\tau K_+^4 \varepsilon  (r-r_+)^{-(4\tau+2b+2)} e^{-\rho \max\{|m|,|n|\}}e^{-|k|r_3}.
\end{equation}
To estimate $\partial_\xi F^{q_m\bar{q}_n}_k$, we differentiate both sides of \eqref{fmn} with respect to $\xi_j (j=1,2,\cdots,b)$ and obtain
\begin{align*}
\partial_{\xi_j} F^{q_m\bar{q}_n}_k=&\frac{{\bf i} \partial_{\xi_j}( P_k^{q_m\bar{q}_n}+\{P^{high}, F^{\theta}+F^1\}_k^{q_m\bar{q}_n})}{\langle k,\omega \rangle +\Omega_m-\Omega_n}\\
&-{\bf i}(P_k^{q_m\bar{q}_n}+ \{P^{high}, F^{\theta}+F^1\}_k^{q_m\bar{q}_n})\cdot \frac{\partial_{\xi_j} (\langle k,\omega \rangle +\Omega_m-\Omega_n)}{(\langle k,\omega \rangle +\Omega_m-\Omega_n)^2}.
\end{align*}
 implies
\begin{align}
\sup_{\xi\in\mathcal{O}_+}|\partial_{\xi_j} F^{q_m\bar{q}_n}_k|\le& \gamma^{-1}|k|^\tau K_+^2 \sup_{\xi\in\mathcal{O}_+}|\partial_\xi   (P_k^{q_m\bar{q}_n}+\{P^{high}, F^{\theta}+F^1\}_k^{q_m\bar{q}_n})|\nonumber\\
&+c\gamma^{-2}|k|^{2\tau+1} K_+^4 \sup_{\xi\in\mathcal{O}_+}|  P_k^{q_m\bar{q}_n}+\{P^{high}, F^{\theta}+F^1\}_k^{q_m\bar{q}_n}|\nonumber\\
\le & c\gamma^{-2}|k|^{2\tau+1} K_+^4 |P_k^{q_m\bar{q}_n}+\{P^{high}, F^{\theta}+F^1\}_k^{q_m\bar{q}_n}|_{\mathcal{O}_+}\nonumber\\
\le &c\gamma^{-6}|k|^{2\tau+1} K_+^{6}  \varepsilon (r-r_+)^{-(4\tau+2b+2)} e^{-\rho \max\{|m|,|n|\}}e^{-|k|r_3}.\label{2ff}
\end{align}
Combining \eqref{1ff} and \eqref{2ff}, one gets for $k\neq 0$,
\begin{equation}\label{Fk11}
|F^{q_m\bar{q}_n}_k|_{\mathcal{O}_+}\le c\gamma^{-6}|k|^{2\tau+1} K_+^{6}  \varepsilon (r-r_+)^{-(4\tau+2b+2)} e^{-\rho \max\{|m|,|n|\}}e^{-|k|r_3}.
\end{equation}
For $k=0$ and $m\neq n$, we have
$$F^{q_m\bar{q}_n}_0=\frac{{\bf i} P_0^{q_m\bar{q}_n}+{\bf i} \{P^{high}, F^{\theta}+F^1\}_0^{q_m\bar{q}_n}}{ \Omega_m-\Omega_n}.$$
Since $\Omega_n=n+\Omega^*_n+O(\varepsilon_0^{\frac{5}{6}}), (|n|\le K_+)$, from \eqref{5.9}, we have
\begin{align*}
|\Omega_m-\Omega_n|\geq \frac{2}{3}-2\varepsilon_0^{\frac{5}{6}}
\geq \frac{1}{2},
\end{align*}
if 
$\varepsilon_0$ is small enough.

Then by similar method of estimating $F^{q_m\bar{q}_n}_k$ for $k\neq 0$, one can easily obtained that
\begin{equation}\label{F011}
|F^{q_m\bar{q}_n}_0|_{\mathcal{O}_+}\le  c\gamma^{-4} K_+^2  \varepsilon (r-r_+)^{-(4\tau+2b+2)} e^{-\rho\max\{|m|, |n|\}}.
\end{equation}
The estimates \eqref{Fk11} and \eqref{F011} implies for all $k\in\mathbb{Z}^b$, $|m|,|n|\le K_+$ and $|k|+|m-n|\neq 0$,
$$|F^{q_m\bar{q}_n}_k|_{\mathcal{O}_+}\le c\gamma^{-6}|k|^{2\tau+1} K_+^{6}  \varepsilon (r-r_+)^{-(4\tau+2b+2)} e^{-\rho \max\{|m|,|n|\}}e^{-|k|r_3}.$$
It follows that
\begin{align*}
\Vert F^{q_m\bar{q}_n}\Vert_{D_2,\mathcal{O}_+} \le& \sum_{k\in \mathbb{Z}^b} |F^{q_m\bar{q}_n}_k|_{\mathcal{O}_+} e^{|k|r_3}\\
\le&  \sum_{k\in \mathbb{Z}^b} c\gamma^{-6}|k|^{2\tau+1} K_+^{6}  \varepsilon (r-r_+)^{-(4\tau+2b+2)} e^{-\rho \max\{|m|,|n|\}}\\
&\cdot e^{-|k|r_3}e^{|k|r_2}\\
\le &c\gamma^{-6}K_+^{6}  \varepsilon (r-r_+)^{-(4\tau+2b+2)} e^{-\rho \max\{|m|,|n|\}}\\
&\cdot \sum_{k\in \mathbb{Z}^b} |k|^{2\tau+1} e^{-|k|\frac{r-r_+}{5}}\\
\le & c\gamma^{-6} K_+^{6}  \varepsilon (r-r_+)^{-(6\tau+3b+3)}e^{-\rho \max\{|m|,|n|\}}.
\end{align*}
\end{proof}
Denote
$$\tilde{D}_{i}=D_{d, \rho_+}\left(r_++\frac{i}{5}(r-r_+), s_++\frac{i}{5}(s-s_+)\right),\quad 1\le i\le 5.$$
The following lemma holds for $X_F$.

\begin{lem}\label{gfpl}
The following estimates hold
\begin{equation}\label{fpfp}
\Vert X_F\Vert_{\tilde{D}_2,\mathcal{O}_+}\le \varepsilon^{\frac{5}{6}}.
\end{equation}
\end{lem}
\begin{proof}
From \eqref{1est} and \eqref{3est}, one deduces that
\begin{equation}\label{xf1}
\frac{1}{s^2}\Vert \partial_\theta F\Vert_{\tilde{D}_2,\mathcal{O}_+}, \quad \Vert \partial_I F\Vert_{\tilde{D}_2,\mathcal{O}_+}\le c\gamma^{-6} K_+^2  \varepsilon (r-r_+)^{-(6\tau+3b+3)},
\end{equation}
and
\begin{align}
&\sup_{\tilde{D}_2}\frac{1}{s}\sum_{n\in\mathbb{Z}}(\Vert \partial_{q_n} F\Vert_{\mathcal{O}_+}+\Vert \partial_{\bar{q}_n} F\Vert_{\mathcal{O}_+})\langle n\rangle^d e^{\rho_+|n|}\nonumber\\
\le &\sup_{\tilde{D}_2}\frac{1}{s}\sum_{|n|\le K_+}(\Vert F^{q_n}\Vert_{\mathcal{O}_+}+\Vert F^{\bar{q}_n}\Vert_{\mathcal{O}_+})\langle n\rangle^d e^{\rho_+|n|}\nonumber\\
&+\sup_{\tilde{D}_2}\frac{1}{s}\sum_{\substack{|m|,|n|\le K_+}}(\Vert F^{q_mq_n}\Vert_{\mathcal{O}_+}+\Vert F^{q_m\bar{q}_n}\Vert_{\mathcal{O}_+}+\Vert F^{\bar{q}_m\bar{q}_n}\Vert_{\mathcal{O}_+})|q_m|\langle n\rangle^d e^{\rho_+|n|}\nonumber\\
\le & c\gamma^{-6} K_+^{6}  \varepsilon (r-r_+)^{-(6\tau+3b+3)}(\rho-\rho_+)^{-d}.\label{xf2}
\end{align}
Combining \eqref{xf1} and \eqref{xf2}, we have there is a constant $c>0$ such that
$$\Vert X_F\Vert_{\tilde{D}_2,\mathcal{O}_+}\le  c\gamma^{-6} K_+^{6}  \varepsilon (r-r_+)^{-(6\tau+3b+3)}(\rho-\rho_+)^{-d}.$$
Then, by \eqref{haoyong} and \eqref{haode}, one gets
$$c\gamma^{-6} K_+^{6}  \varepsilon (r-r_+)^{-(6\tau+3b+3)}(\rho-\rho_+)^{-d}\varepsilon^{\frac{1}{6}}\le 1,
$$
by choosing $\varepsilon_0$ small enough.
 Then estimates \eqref{fpfp} follows immediately.
\end{proof}

\begin{lem}\label{reg}
The flow map $X_F^t$ statisfies  $X_F^t: \tilde{D}_{1}\rightarrow \tilde{D}_{2}, -1\le t\le 1$. In addition, the following estimate holds
$$\Vert X^t_F-id\Vert_{D_+}\le \varepsilon^{\frac{5}{6}},\quad \Vert \mathcal{D}X^t_F-Id\Vert_{D_+}\le 2\varepsilon^{\frac{4}{5}},$$\
where $D_+=D_{d,\rho_+}(r_+,s_+)$.
\end{lem}
\begin{proof}
Let
$$\Vert \mathcal{D}^m F\Vert_{D,\mathcal{O}} =\max\left\{\left\Vert \frac{\partial^{|i|+|l|+|\alpha|+|\beta|}F}{\partial\theta^i\partial I^l\partial q^\alpha\partial \bar{q}^\beta}\right\Vert_{D,\mathcal{O}_+}, \ |i|+|l|+|\alpha|+|\beta|=m\geq 2\right\}.$$
Since $F$ is a polynomial of order $1$ in $I$ and of order $2$ in $q$ and $\bar{q}$, by Lemma \ref{gfpl} and Lemma \ref{Cauchy}, we obtain
$$\Vert \mathcal{D}^m F\Vert_{\tilde{D}_{1},\mathcal{O}_+}\le \varepsilon^{\frac{4}{5}},\quad \forall m\geq 2.$$
Noting the equation
$$X_F^t=id+\int_0^t X_F\circ X_F^s ds$$
and Lemma \ref{gfpl}, one derives that $X_F^t:\tilde{D}_{1}\rightarrow \tilde{D}_{2}, -1\le t\le 1$ with
$$\Vert X^t_F-id\Vert_{D_+}\le \Vert X_F\Vert_{D_+}\le  \varepsilon^{\frac{5}{6}}.$$
In addition, we have the integral equation
\begin{align*}
\mathcal{D}X_F^t=&Id+\int_0^t \mathcal{D}X_F \cdot\mathcal{D} X_F^s ds\\
=&Id+\int_0^t J(\mathcal{D}^2F )\mathcal{D} X_F^s ds,
\end{align*}
where the notation $J$ is the standard symplectic matrix. It implies that
$$\Vert \mathcal{D}X_F^t-Id\Vert_{D_{+}}\le 2\Vert \mathcal{D}^2F\Vert_{\tilde{D}_{1}}\le 2\varepsilon^{\frac{4}{5}}.$$
\end{proof}

\subsection{Estimate for the new Hamiltonian}
Let $\Phi=X_F^{1}$ and
\begin{equation}\label{nnn}
N_+=e_++\langle \omega_+,I\rangle+\sum_{n\in\mathbb{Z}_1}\Omega^+_nq_n\bar{q}_n,
\end{equation}
where
\begin{align*}
e_+=e+\hat{e},\quad  \omega_+=\omega+\hat{\omega},\quad \Omega^+_n=\Omega_n+\hat{\Omega}_n,
\end{align*}
where $\hat{e},\ \hat{\omega}$ and $\hat{\Omega}_n$ is defined by \eqref{nN}. Combining Lemma \ref{holemma} and \eqref{ohat}, \eqref{Ohat} and \eqref{nN}, we obtain
$$|\omega_+-\omega|\le \varepsilon^{\frac{5}{6}},\quad |\Omega^+_n-\Omega_n|\le \varepsilon^{\frac{5}{6}}e^{-\rho |n|}.$$

In addition, by subsection \ref{sub5.1}, we have
\begin{align}
P_+=&\hat{P}+P^{high}+\{P^{high},F\}^{high}\nonumber\\
&+\int_0^1 (1-t)\{\{N+P^{high},F\},F\}\circ X_F^tdt\nonumber\\
&+\int_0^1 \{P^{low},F\}\circ X_F^tdt.\label{tusi}
\end{align}
where $\hat{P}$ is defined by \eqref{oOme}.

By the definition of $\hat{P}$ and \eqref{haoy1}, \eqref{haoy2}, one gets
\begin{equation}\label{phat}
\Vert X_{\hat{P}}\Vert_{\tilde{D}_{2},\mathcal{O}_+} \le c\gamma^{-4} K_+^2  \varepsilon (r-r_+)^{-(4\tau+2b+2)} \sum_{|n|> K_+} e^{-(\rho-\rho_+)|n|}\le \varepsilon^{\frac{5}{4}}.
\end{equation}

Note
\begin{align}\label{NF}
\{N,F\}=\hat{N}+\hat{P}-P^{low}-\{P^{high},F\}^{low},
\end{align}
which satisfies that $\Vert X_{\{N,F\}}\Vert_{\tilde{D}_1,\mathcal{O}_+}\le \varepsilon^{\frac{5}{6}}.$ Then, by Lemma \ref{gfpl}, we have
\begin{equation}\label{low1}
\Vert X_{\{\{N,F\},F\}}\Vert_{\tilde{D}_1,\mathcal{O}_+} \le c\varepsilon^{\frac{5}{3}}.
\end{equation}
Similarly, one can obtain
\begin{equation}\label{low2}
\Vert X_{\{P^{low},F\}}\Vert_{\tilde{D}_1,\mathcal{O}_+} \le c\varepsilon^{\frac{11}{6}},\quad\Vert X_{\{\{P^{high},F\},F\}}\Vert_{\tilde{D}_1,\mathcal{O}_+} \le c\varepsilon^{\frac{5}{3}}.
\end{equation}
In view of formula \eqref{tusi}, we get
\begin{align*}
P_+^{low}=&\hat{P}+\left(\int_0^1 (1-t)\{\{N+P^{high},F\},F\}\circ X_F^tdt\right)^{low}\\
&\quad +\left(\int_0^1 \{P^{low},F\}\circ X_F^tdt\right)^{low},\\
P_+^{high}=&P^{high}+\{P^{high},F\}^{high}\\
&\quad  \quad  \ +\left(\int_0^1 (1-t)\{\{N+P^{high},F\},F\}\circ X_F^tdt\right)^{high}\\
&\quad \quad \  +\left(\int_0^1 \{P^{low},F\}\circ X_F^tdt\right)^{high}.
\end{align*}

From Lemma \ref{reg} and estimates \eqref{phat}, \eqref{low1} and \eqref{low2}, we get
$$\Vert X_{P^{low}_+}\Vert_{D_+,\mathcal{O}_+}\le c\varepsilon^{\frac{5}{4}},$$
and
$$\Vert X_{P^{high}_+}\Vert_{D_+,\mathcal{O}_+}\le \Vert X_{P^{high}}\Vert_{D_+,\mathcal{O}_+}+c \varepsilon^{\frac{5}{6}}\le 2.$$

It is easy to verify the gauge invariance are maintained in the iterative process, since we have prove the Poisson bracket keep the gauge invariance.

Next, we will verify $P_+$ satisfies decay property. Let
$$P_+^{low}=\sum_{\substack{\alpha,\beta\in \mathbb{N}^{\mathbb{Z}_1}}}P^{low}_{+,\alpha\beta}(\theta,I;\xi)q^\alpha \bar{q}^\beta,$$
and
$$\ddot{P}_+=P_+^{high}-\dot{P}=\sum_{\alpha,\beta\in \mathbb{N}^{\mathbb{Z}_1}}\ddot{P}^+_{\alpha\beta}(\theta, I;\xi)q^\alpha\bar{q}^\beta.$$
where the term $\dot{P}$ remains unchanged during the whole KAM iterative and it has the same decay property as it in the beginning. The term $\dot{P}\circ \Phi-\dot{P}$ will be put into $\ddot{P}_+$ and $P_+^{low}$.
Thus, for $\dot{P}=\sum_{\alpha,\beta\in \mathbb{N}^{\mathbb{Z}_1}}\dot{P}_{\alpha\beta}(\xi)q^\alpha\bar{q}^\beta$, from
$$\Vert \dot{P}_{\alpha\beta}\Vert_{D,\mathcal{O}}\le
e^{-\rho(n_{\alpha\beta}^+-n_{\alpha\beta}^-)},$$
one sees
$$\Vert \dot{P}_{\alpha\beta}\Vert_{D_+,\mathcal{O}_+}\le
e^{-\rho_+(n_{\alpha\beta}^+-n_{\alpha\beta}^-)}.$$
In what follows, we will show
$$
\Vert P^{low}_{+,\alpha\beta}\Vert_{D_+,\mathcal{O}_+}\le
\varepsilon_+ e^{-\rho_+ n^*_{\alpha\beta}},\quad
\Vert \ddot{P}^+_{\alpha\beta}\Vert_{D_+,\mathcal{O}_+}\le
e^{-\rho_+ n^*_{\alpha\beta}}.
$$

Firstly, we have
\begin{align*}
P_+^{low}+\ddot{P}_+=&\hat{P}+(P^{high}-\ddot{P})+\{P^{high},F\}^{high}+\{P^{low},F\}\\
&\quad +\frac{1}{2!}\{\{N,F\},F\}+\frac{1}{2!}\{\{P,F\},F\}+\cdots\\
&\quad +\frac{1}{n!}\{\cdots\{N,F\},\cdots F\}+\frac{1}{n!}\{\cdots\{P,F\},\cdots F\}\\
&\quad +\cdots.
\end{align*}
For the term $\hat{P}$, from \eqref{pb}, \eqref{pa}, \eqref{1est}, \eqref{2est} and the definition of $\hat{P}$ in \eqref{oOme}, for $n^*_{\alpha\beta}>K_+$, one gets
\begin{align*}
\Vert \hat{P}_{+,\alpha\beta}\Vert_{D_+,\mathcal{O}_+}\le& c\gamma^{-5} K_+^8 \varepsilon  (r-r_+)^{-(4\tau+2b+2)} e^{-\rho n^*_{\alpha\beta}}\\
\le& c\gamma^{-5} K_+^8 \varepsilon  (r-r_+)^{-(4\tau+2b+2)} e^{-(\rho-\rho_+)n^*_{\alpha\beta}}\cdot e^{-\rho_+ n^*_{\alpha\beta}} \\
\le &\varepsilon_+ e^{-\rho_+ n^*_{\alpha\beta}}
\end{align*}
if $\varepsilon_0$ is sufficiently small.

For the term $P^{high}-\dot{P}$, we have
$$\Vert P^{high}-\dot{P}\Vert_{D_+,\mathcal{O}_+}=\Vert \ddot{P}_{\alpha\beta}\Vert_{D_+,\mathcal{O}_+}\le
2e^{-\rho_+ n^*_{\alpha\beta}}.$$

The decay property of the remainer in $P_+$, which are made up of several Poisson brackets, can be obtained by the following lemmata.
\begin{lem}\label{5.5}
The following estimate hold
\begin{align}\label{lem5.5}
\Vert \{P^{high},F\}_{\alpha \beta}\Vert_{\tilde{D}_{1},\mathcal{O}_+}\le c\varepsilon^{\frac{4}{5}}e^{-\rho_+  n^*_{\alpha\beta}}.
\end{align}
\end{lem}
\begin{proof}
By the definition of $P^{high}$, we have
$$P^{high}=\dot{P}+\ddot{P},$$
with
$$\Vert \dot{P}_{\alpha\beta}\Vert_{D,\mathcal{O}}\le
e^{-\rho(n_{\alpha\beta}^+-n_{\alpha\beta}^-)},\quad \Vert \ddot{P}_{\alpha\beta}\Vert_{D,\mathcal{O}}\le
e^{-\rho n^*_{\alpha\beta}}.$$
By simple computation, one has
\begin{align*}
\{\dot{P},F\}_{\alpha\beta}=&{\bf i} \sum_{\substack{n\in \mathbb{Z}\\ (l,k)+(L,K)=(\alpha,\beta)}}(\dot{P}_{l+e_n,k}F_{L,K+e_n}-\dot{P}_{l,k+e_n}F_{L+e_n,K}).
\end{align*}
It implies that
\begin{align}\label{ruiq1}
\Vert \dot{P}_{l+e_n,k}F_{L,K+e_n}\Vert_{\tilde{D}_1,\mathcal{O}_+}\le& e^{-\rho(n^+_{l+e_n,k}-n^-_{l+e_n,k})}\cdot \varepsilon^{\frac{5}{6}} e^{-\rho_+ n^*_{L,K+e_n}}\nonumber\\
\le & \varepsilon^{\frac{5}{6}} e^{-\rho_+ n_{\alpha\beta}^*},
\end{align}
by noticing that $n_{\alpha\beta}^*\le n^+_{l+e_n,k}-n^-_{l+e_n,k}+n^*_{L,K+e_n}$. Similarly, one can prove
\begin{align}\label{ruiq2}
\Vert \dot{P}_{l,k+e_n}F_{L+e_n,K})\Vert_{\tilde{D}_1,\mathcal{O}_+}\le & \varepsilon^{\frac{5}{6}} e^{-\rho_+n_{\alpha\beta}^*},
\end{align}
From \eqref{ruiq1} and \eqref{ruiq2}, we have
\begin{align}\label{zbu1}
\Vert \{\dot{P},F\}_{\alpha\beta}\Vert_{\tilde{D}_1,\mathcal{O}_+}\le  c\varepsilon^{\frac{5}{6}} e^{-\rho_+n_{\alpha\beta}^*}.
\end{align}
For the term $\{\ddot{P},F\}$, we have
\begin{align*}
\{\ddot{P},F\}_{\alpha\beta}=&{\bf i} \sum_{\substack{n\in \mathbb{Z}\\ (l,k)+(L,K)=(\alpha,\beta)}}(\ddot{P}_{l+e_n,k}F_{L,K+e_n}-\ddot{P}_{l,k+e_n}F_{L+e_n,K})\\
&+\sum_{(l,k)+(L,K)=(\alpha,\beta)}\{\ddot{P}_{lk},F_{LK}\}.
\end{align*}
By the fact $n_{\alpha\beta}^*\le \max\{ n^*_{l+e_n,k}, n^*_{L,K+e_n}\}$, we have
\begin{align}\label{hd1}
\Vert \ddot{P}_{l+e_n,k}F_{L,K+e_n}\Vert_{\tilde{D}_1,\mathcal{O}_+}\le &e^{-\rho n^*_{l+e_n,k}}\cdot \varepsilon^{\frac{5}{6}} e^{-\rho_+n^*_{L,K+e_n}}\nonumber\\
\le  &\varepsilon^{\frac{5}{6}} e^{-\rho_+n_{\alpha\beta}^*}.
\end{align}
Similarly, one has
\begin{align}\label{hd2}
\Vert \ddot{P}_{l,k+e_n}F_{L+e_n,K}\Vert_{\tilde{D}_1,\mathcal{O}_+}
\le  &\varepsilon^{\frac{5}{6}} e^{-\rho_+n_{\alpha\beta}^*}.
\end{align}
From Lemma \ref{czlem} and the inequality $n_{\alpha\beta}^*\le \max\{ n^*_{l,k}, n^*_{L,K}\}$, we have
\begin{equation}\label{hd3}
\Vert \{\ddot{P}_{lk},F_{LK}\}_{\alpha\beta}\Vert_{\tilde{D}_{1},\mathcal{O}_+}\le c(r-r_+)^{-1}(s-s_+)^{-2}\varepsilon^{\frac{5}{6}}e^{-\rho_+ n^*_{\alpha\beta}}\le\varepsilon^{\frac{4}{5}} e^{-\rho_+n_{\alpha\beta}^*}.
\end{equation}
By the previous estimates \eqref{hd1}, \eqref{hd2} and \eqref{hd3}, one obtains
\begin{align}\label{zbu2}
\Vert \{\ddot{P},F\}_{\alpha\beta}\Vert_{\tilde{D}_1,\mathcal{O}_+}\le  c\varepsilon^{\frac{4}{5}} e^{-\rho_+n_{\alpha\beta}^*}.
\end{align}
Putting estimates \eqref{zbu1} and \eqref{zbu2} together, we show estimate \eqref{lem5.5}.
\end{proof}

\begin{lem}\label{5.6}
The following estimate hold
\begin{align*}
\Vert \{P^{low},F\}_{\alpha \beta}\Vert_{\tilde{D}_{1},\mathcal{O}_+}\le \varepsilon^{\frac{9}{5}}e^{-\rho  n^*_{\alpha\beta}}.
\end{align*}
\end{lem}
The proof of this lemma is similar to the estimate of the term $\{\ddot{P},F\}$ and nothing is new.

From Lemma \ref{5.5} and Lemma \ref{5.6}, we immediately have
\begin{lem}\label{huax}
The following estimates hold
$$
\Vert \{P,F\}_{\alpha\beta}\Vert_{\tilde{D}_{1},\mathcal{O}_+}\le \varepsilon^{\frac{4}{5}}e^{-\rho n^*_{\alpha\beta}},$$
\end{lem}

From \eqref{NF}, we deduce that
$$\Vert \{N,F\}_{\alpha\beta}\Vert_{\tilde{D}_1,\mathcal{O}_+}\le \varepsilon^{\frac{5}{6}} e^{-\rho n_{\alpha\beta}^*}.$$
Then, we have the following lemma and the proof of this lemma is similar to Lemma \ref{5.6}.
\begin{lem}
The following estimates hold
$$\Vert \{\{N,F\},F\}\Vert_{\tilde{D}_{1},\mathcal{O}_+}\le \frac{1}{4}\varepsilon^{\frac{5}{4}}e^{-\rho_+n_{\alpha\beta}^*}.$$
\end{lem}

Summarize the analysis above and let
\begin{align*}
P_+^{low}=&\hat{P}+\{P^{low},F\}\\
&\quad +\frac{1}{2!}\{\{N,F\},F\}^{low}+\frac{1}{2!}\{\{P,F\},F\}^{low}+\cdots\\
&\quad +\frac{1}{n!}\{\cdots\{N,F\},\cdots F\}^{low}+\frac{1}{n!}\{\cdots\{P,F\},\cdots F\}^{low}\\
&\quad +\cdots.
\end{align*}
and
\begin{align*}
\ddot{P}_+=&(P^{high}-\dot{P})+\{P^{high},F\}^{high}\\
&+\frac{1}{2!}\{\{N,F\},F\}^{high}+\frac{1}{2!}\{\{P,F\},F\}^{high}+\cdots\\
&+\frac{1}{n!}\{\cdots\{N,F\},\cdots F\}^{high}+\frac{1}{n!}\{\cdots\{P,F\},\cdots F\}^{high}\\
&+\cdots.
\end{align*}
then the decay property for $P^{low}_+$ and $\ddot{P}_+$ can be expressed as the following
\begin{lem}
The new perturbation $$P_+^{low}=\sum_{\substack{\alpha,\beta\in \mathbb{N}^{\mathbb{Z}_1}}}P^{low}_{+,\alpha\beta}(\theta,I;\xi)q^\alpha \bar{q}^\beta,$$
and
$$\ddot{P}_+=\sum_{\alpha,\beta\in \mathbb{N}^{\mathbb{Z}_1}}\ddot{P}^+_{\alpha\beta}(\theta, I;\xi)q^\alpha\bar{q}^\beta,$$ satisfies
$$
\Vert P^{low}_{+,\alpha\beta}\Vert_{D_+,\mathcal{O}_+}\le
\varepsilon_+ e^{-\rho_+ n^*_{\alpha\beta}},$$
$$
\Vert \ddot{P}^+_{\alpha\beta}\Vert_{D_+,\mathcal{O}_+}\le  \Vert \ddot{P}_{\alpha\beta}\Vert_{D_+,\mathcal{O}_+}+c\varepsilon^{\frac{4}{5}}e^{-\rho_+ n^*_{\alpha\beta}}\le
2e^{-\rho_+ n^*_{\alpha\beta}}.
$$
\end{lem}

\subsection{The convergence of KAM iteration }
In the following, we will prove the iteration process is converge. Let us define $\Phi^\nu=\Phi_1\circ\Phi_2\circ\cdots \circ \Phi_{\nu}, \nu=1,2,\cdots$. By induction argument shows that $\Phi^{\nu}: D_{\nu}\times \mathcal{O}_\nu\rightarrow D_{\nu-1}\times \mathcal{O}_{\nu-1}$ and
$$H_0\circ\Phi^\nu=H_\nu=N_\nu+P_\nu.$$
Let $\mathcal{O}_\varepsilon=\cap_{\nu=0}^\infty \mathcal{O}_\nu$. 
In view of Lemma \ref{itlem}, it concludes that for $\nu\geq 1$
\begin{align*}
\Vert \Phi^{\nu+1}-\Phi^{\nu}\Vert_{D_{\nu+1}\times \mathcal{O}_{\nu+1}}\le \Vert D\Phi^{\nu}\Vert _{D_{\nu}\times \mathcal{O}_{\nu}} \Vert \Phi_{\nu+1}-id\Vert_{D_{\nu+1}\times \mathcal{O}_{\nu+1}}
\end{align*}
and
\begin{align*}
\Vert D\Phi^{\nu}\Vert _{D_{\nu}\times \mathcal{O}_{\nu}} \le \prod_{i=1}^{\nu} \Vert D\Phi_{i}\Vert _{D_{i}\times \mathcal{O}_{i}}\le  \prod_{i=1}^{\nu} (1+c\varepsilon_{i-1}^{\frac{4}{5}})\le 2.
\end{align*}
It follows that
\begin{align*}
\Vert \Phi^{\nu+1}-\Phi^{\nu}\Vert_{D_{\nu+1}\times \mathcal{O}_{\nu+1}}\le 2 \Vert \Phi_{\nu+1}-id\Vert_{D_{\nu+1}\times \mathcal{O}_{\nu+1}}\le c\varepsilon_{\nu}^{\frac{5}{6}}.
\end{align*}
Hence, let $\Phi^0=id$, one gets
\begin{align*}
\Vert \Phi^{\nu+1}-\Phi^0\Vert_{D_{\nu+1}\times \mathcal{O}_{\nu+1}}\le \sum_{i=0}^{\nu} \Vert \Phi^{i+1}-\Phi^{i}\Vert_{D_{\nu+1}\times \mathcal{O}_{\nu+1}} \le \sum_{i=0}^{\nu} c\varepsilon_{\nu}^{\frac{5}{6}}\le c\varepsilon_{0}^{\frac{5}{6}}
\end{align*}
It means that $\Phi^\nu$ converge uniformly on $D_{d,0}(\frac{1}{2}r_0,\frac{1}{2}s_0)\times \mathcal{O}_\varepsilon$ to $\Phi^\infty$
with
\begin{equation}\label{phiinfty}
\Vert \Phi^\infty-\Phi^0\Vert_{D_{d,0}(\frac{1}{2}r_0,\frac{1}{2}s_0)\times \mathcal{O}_\varepsilon}\le c\varepsilon_0^{\frac{5}{6}}.
\end{equation}
Moreover, The Hamiltonions $H_\nu$ will converge with
$$N_\infty=e_\infty+\langle \omega_\infty, I\rangle+\sum_{n\in\mathbb{Z}_1}\Omega^\infty_n q_n\bar{q}_n.$$
Since
$$\varepsilon_{\nu}=c\varepsilon_{\nu-1}^{\frac{5}{4}}=c^{-4}(c^4\varepsilon_0)^{(\frac{5}{4})^{\nu}},$$
 we have $\varepsilon_\nu\rightarrow 0$ as $\nu\rightarrow \infty$. By noting  $\Vert X_{P_\nu}\Vert_{D_{d,0}(\frac{1}{2}r_0,\frac{1}{2}s_0),\mathcal{O}_\varepsilon}\le \varepsilon_\nu$, it follows that $\Vert X_{P_\nu}\Vert_{D_{d,0}(\frac{1}{2}r_0,\frac{1}{2}s_0),\mathcal{O}_\varepsilon}\rightarrow 0$ as $\nu\rightarrow 0$.

Let $\phi_{H_0}^t$ be the flow of $X_{H_0}$. From $H_0\circ \Phi^\nu=H_\nu$, we know that
\begin{equation}\label{phih}
\phi_{H_0}^t\circ \Phi^\nu=\Phi^\nu\circ\phi_{H_\nu}^t.
\end{equation}
The uniform convergence of $\Phi^\nu$ and $X_{H_\nu}$ implies that one can take the limit in \eqref{phih} and arrive at
\begin{equation}\label{phih1}
\phi_{H_0}^t\circ \Phi^\infty=\Phi^\infty\circ\phi_{H_\infty}^t
\end{equation}
on $D_{d,0}(\frac{1}{2}r_0,\frac{1}{2}s_0)\times\mathcal{O}_\varepsilon$, where
$$\Phi^\infty:D_{d,0}(\frac{1}{2}r_0,\frac{1}{2}s_0)\times\mathcal{O}_\varepsilon\rightarrow D_0\times \mathcal{O}_0.$$
It follows from \eqref{phih1} that
$$\phi_{H_0}^t(\Phi^\infty(\{\xi\}\times \mathbb{T}^b)=\Phi^\infty\phi_{H_\infty}^t(\{\xi\}\times \mathbb{T}^b)=\Phi^\infty(\{\xi\}\times \mathbb{T}^b)$$
for $\xi\in\mathcal{O}_\varepsilon$. This means that $\Phi^\infty(\{\xi\}\times \mathbb{T}^b)$ is an embedded invariant torus of the original perturbed Hamiltonian system at $\xi\in\mathcal{O}_\varepsilon$ with the estimate \eqref{phiinfty}. Moreover, we remark that the frequencies $\omega_\infty(\xi)$ associated with $\Phi^\infty(\{\xi\}\times \mathbb{T}^b)$ are slightly different from $\omega^0(\xi)$. Actually, by Lemma \ref{itlem}, we have
\begin{align*}
\Vert \omega^{\nu+1}-\omega^0\Vert_{\mathcal{O}_{\nu+1}}\le\sum_{i=0}^{\nu}\Vert   \omega^{i+1}-\omega^i\Vert_{\mathcal{O}_{i+1}}\le \sum_{i=0}^\nu \varepsilon_i^{\frac{5}{6}}\le c\varepsilon_0^{\frac{5}{6}}.
\end{align*}
Therefore, by letting $\nu\rightarrow \infty$, we get
$$\Vert \omega_\infty-\omega^0\Vert_{\mathcal{O}_\varepsilon}\le c\varepsilon_0^{\frac{5}{6}}.$$
\subsection{measure estimate}
In the rest, we will estimate the measure for the parameters. By the iteration lemma
$$\mathcal{O} \backslash \mathcal{O}_\varepsilon\subset \cup_{\nu=0}^\infty \cup_{k \neq 0}\mathcal{R}_k^\nu$$
with
$$\mathcal{R}_k^\nu=\mathcal{R}_k^{\nu0}\bigcup \left(\bigcup_{|n|\le K_{\nu+1}}\mathcal{R}_{kn}^{\nu 1}\right)\bigcup \left(\bigcup_{|m|,|n|\le K_{\nu+1}}\mathcal{R}_{km}^{\nu2}\right)\bigcup \left(\bigcup_{|m|,|n|\le K_{\nu+1}}\mathcal{R}_{kmn}^{\nu3}\right),$$
where
\begin{align*}
&\mathcal{R}_k^{\nu0}=\left\{\xi\in\mathcal{O}_\nu: |\langle k,\omega_\nu\rangle|<\frac{\gamma_\nu}{|k|^\tau}\right\},\\
&\mathcal{R}_{kn}^{\nu1}=\left\{\xi\in\mathcal{O}_\nu: |\langle k,\omega_\nu\rangle+\Omega_n^\nu|<\frac{\gamma_\nu}{|k|^\tau K_{\nu+1}}\right\},\\
&\mathcal{R}_{kmn}^{\nu2}=\left\{\xi\in\mathcal{O}_\nu: |\langle k,\omega_\nu\rangle+\Omega_m^\nu+\Omega_n^\nu|<\frac{\gamma_\nu}{|k|^\tau K^{2}_{\nu+1}}\right\},\\
&\mathcal{R}_{kmn}^{\nu3}=\left\{\xi\in\mathcal{O}_\nu: |\langle k,\omega_\nu\rangle+\Omega_m^\nu-\Omega_n^\nu|<\frac{\gamma_\nu}{|k|^\tau K^{2}_{\nu+1}}\right\}.
\end{align*}





For $|k|\neq 0$, without loss of generality we suppose $|k_1|=\max\{ |k_1|,|k_2|\cdots,|k_b|\}$. From assumption {\bf  Assumption A} and {\bf   B}, one has
\begin{align*}
&|\partial_{\xi_1}(\langle k,\omega_\nu\rangle +\Omega_m^\nu-\Omega_n^\nu)|\\
 \geq & |k_1|(1-\frac{1}{4b}-2\varepsilon_0^{\frac{5}{6}})-(|k|-|k_1|)(\frac{1}{4b}+2\varepsilon_0^{\frac{5}{6}})-\frac{1}{2b}-4\varepsilon_0^{\frac{5}{6}}\\
\geq &|k_1|-|k|(\frac{1}{4b}+2\varepsilon_0^{\frac{5}{6}})-\frac{1}{2b}-4\varepsilon_0^{\frac{5}{6}}\\
\geq &\frac{1}{6b}|k|,
\end{align*}
for sufficiently small $\varepsilon_0$ . Therefore, 
we have
$$|\mathcal{R}_{kmn}^{\nu3}|\le c\gamma_\nu |k|^{-(\tau+1)}K_{\nu+1}^{-2}\cdot |\mathcal{O}|.$$
where $c$ is a constant depend on $b$.  By similar discussion, we will obtain
\begin{align*}
&\left|\mathcal{R}_k^{\nu0}\bigcup \left(\bigcup_{|n|\le K_{\nu+1}}\mathcal{R}_{kn}^{\nu 1}\right)\bigcup \left(\bigcup_{|m|,|n|\le K_{\nu+1}}\mathcal{R}_{km}^{\nu2}\right)\bigcup \left(\bigcup_{|m|,|n|\le K_{\nu+1}}\mathcal{R}_{kmn}^{\nu3}\right)\right| \\
\le &c\gamma_\nu|k|^{-(\tau+1)} |\mathcal{O}|.
\end{align*}
Thus, we have
$$|\mathcal{O}\setminus \mathcal{O}_\varepsilon| \le |\bigcup_{\nu\geq 0}\bigcup_{k\neq 0} \mathcal{R}_k^\nu |\le c\sum_{\nu\geq 0}\sum_{k\neq 0}\frac{\gamma_\nu}{|k|^{\tau+1}}=c\sum_{\nu\geq 0}\gamma_\nu\cdot |\mathcal{O}|\sim \gamma_0 |\mathcal{O}|=\varepsilon^{\frac{1}{16}} |\mathcal{O}|.$$
by $\tau>b-1$.

\section*{Declarations}

\textbf{Funding} Shengqing Hu is supported by [National Natural Science Foundation of China (Grant No.12201532)]

\textbf{Conflict of interest}  On behalf of all authors, the corresponding author states that there is no conflict of interest.\\

\textbf{Availability of data }  The manuscript has no associated data.

\bibliographystyle{alpha} 
\bibliography{ref}

 \end{document}